\documentclass[authoryear,nopreprintline]{elsarticle}
\usepackage[utf8]{inputenc}
\usepackage{amsmath}
\usepackage{mathtools} 
\usepackage{amssymb}
\usepackage{amsthm}
\usepackage{tikz}
\usepackage{xcolor}
\usepackage{hyperref}
\usepackage{enumitem}
\usepackage{url}
\usepackage{xfrac}
\usepackage{natbib}
\usepackage{bm}
\usepackage{mathabx}
\usepackage{algorithm}
\usepackage{algpseudocode}
\usepackage{graphicx}

\theoremstyle{plain}
\newtheorem{lemma}{Lemma}[section]
\newtheorem{proposition}[lemma]{Proposition}
\newtheorem{corollary}[lemma]{Corollary}

\newtheorem{theorem}[lemma]{Theorem}

\theoremstyle{definition}

\newtheorem{remark}[lemma]{Remark}
\newtheorem{example}[lemma]{Example}
\newtheorem{assumption}[lemma]{Assumption}

\newcommand{\R}{\mathbb{R}}

\newcommand{\N}{\mathbb{N}}

\newcommand{\dd}{\textnormal{d}}
\newcommand{\bx}{\mathbf{x}}
\newcommand{\by}{\mathbf{y}}

\newcommand{\bZ}{\mathbf{Z}}
\newcommand{\bR}{\mathbf{R}}
\newcommand{\bz}{\mathbf{z}}
\newcommand{\bp}{\mathbf{p}}
\newcommand{\bq}{\mathbf{q}}
\newcommand{\br}{\mathbf{r}}
\newcommand{\be}{\mathbf{e}}
\newcommand{\bzero}{\mathbf{0}}
\newcommand{\bone}{\mathbf{1}}

\newcommand{\blambda}{\bm{\lambda}}
\newcommand{\bmu}{\bm{\mu}}
\newcommand{\nnR}{\R_{\geq 0}}
\newcommand{\bulletized}[1]{{#1}_{\bullet \bullet}}  
\newcommand{\bulletize}[1]{{#1}_{\bullet}}
\newcommand{\kapemp}{\bar{I}_{\kappa} }

\newcommand{\typ}{D}

\DeclareSymbolFont{bbold}{U}{bbold}{m}{n}
\DeclareSymbolFontAlphabet{\mathbbold}{bbold}
\newcommand{\ind}{\mathbbold{1}}

\renewcommand{\P}{\mathbb{P}}
\newcommand{\E}{\mathbb{E}}

\newcommand{\Mp}{\mathcal{M}_1}
\newcommand{\Mpn}{\mathcal{M}_{1,N}}

\newcommand{\supp}{\mathrm{supp}}

\newcommand{\birth}{b}
\newcommand{\death}{d}
\newcommand{\mut}{m}

\begin{document}

\begin{frontmatter}
	\title{Mean-field interacting multi-type birth-death processes with a view to applications in phylodynamics}

	\author[1]{William S. DeWitt}
	\ead{wsdewitt@berkeley.edu}
	\author[2]{Steven N. Evans\corref{cor1}}
	\ead{evans@stat.berkeley.edu}
	\author[2]{Ella Hiesmayr}
	\ead{ella.hiesmayr@berkeley.edu}
	\author[2]{Sebastian Hummel}
	\ead{shummel@berkeley.edu}

	\affiliation[1]{organization={Department of Electrical Engineering \& Computer Sciences, University of California, Berkeley}}

	\affiliation[2]{organization={Department of Statistics, University of California, Berkeley}}

	\cortext[cor1]{Corresponding author}

	\begin{abstract}
	Multi-type birth-death processes underlie approaches for inferring evolutionary dynamics from phylogenetic trees across biological scales, ranging from deep-time species macroevolution to rapid viral evolution and somatic cellular proliferation.
	A limitation of current phylogenetic birth-death models is that they require restrictive linearity assumptions that yield tractable message-passing likelihoods, but that also preclude interactions between individuals.
	Many fundamental evolutionary processes---such as environmental carrying capacity or frequency-dependent selection---entail interactions, and may strongly influence the dynamics in some systems.
	Here, we introduce a multi-type birth-death process in mean-field interaction with an ensemble of replicas of the focal process.
	We prove that, under quite general conditions, the ensemble's stochastically evolving interaction field converges to a \emph{deterministic} trajectory in the limit of an infinite ensemble.
	In this limit, the replicas effectively decouple, and self-consistent interactions appear as nonlinearities in the infinitesimal generator of the focal process.
	We investigate a special case that is rich enough to model both carrying capacity and frequency-dependent selection while yielding tractable message-passing likelihoods in the context of a phylogenetic birth-death model.
	\end{abstract}

\end{frontmatter}

\section{Introduction}

\subsection{The multi-type birth-death process}

The \emph{multi-type birth-death process} (MTBDP) is a continuous-time Markov chain generalizing the classical birth-death process \citep{feller1968introduction,Kendall1948} to a finite number of \emph{types}.
The state of the MTBDP counts the number of individuals (or \emph{particles}) of each type while they undergo birth, death, and type transition events according to specified rates, which may be arbitrary functions of the current state and of time.
If these rates are linear in the state, the MTBDP can be formulated as a branching process \citep{Griffiths1973}.
If additionally, the rates for each type are proportional to the count of only that type, the MTBDP is said to be \emph{simple}, and the rates can be specified particle-wise because particles do not interact.
The general case of nonlinear rates has also been called a \emph{multivariate competition process} \citep{reuter1961competition, Iglehart1964}, which, as noted by \citet{ho2018birth}, is more restrictive than a multi-type branching process in that the latter allows for increments other than unity, and more general in that the latter is manifestly linear via its defining independence property.

\subsection{Phylogenetic birth-death models}

The MTBDP has facilitated the inference of diversification processes in biological systems, with applications ranging across scales of evolutionary time and biological organization.
\emph{Phylogenetic birth-death} models assume that a phylogenetic tree is generated by an MTBDP combined with a sampling process that censors subtrees that are not ancestral to any sampled leaves, so that histories are only partially observed.
The diverse flavors of these models are reviewed and introduced with unified notation in \citet{MacPherson2021}.
Given a phylogeny, the inferential targets are the birth and death rates, as well as the type transition rates.
Birth and death are variously interpreted as extinction and speciation rates in the context of macroevolutionary studies, or as transmission and recovery rates in the context of epidemiological or viral phylodynamic studies.
The literature contains many variants of this modeling approach.
Depending on the application, the birth and death rates may be assumed to be time-dependent, depend on particle type, or both.

To facilitate tractable likelihoods, phylogenetic birth-death models assume the restrictive non-interacting simple MTBDP, with particle-wise birth and death rates that depend only on particle type, and possibly on time.
In this case, given a time-calibrated tree, the likelihood---defined via the conditional density of the tree assuming it has at least one sampled descendant---can be evaluated via tree message-passing computations.
This message-passing structure can be seen to follow from elementary properties of branching processes, adapted to partial tree observation.
The message functions \cite[in work by][these are called \emph{branch propagators}]{neher2014predicting} are given by the solutions to master equations that marginalize over all possible unobserved subtrees subtending the branch, and are computed recursively via post-order traversal (from tree tips to root).

\subsection{Biology involves interactions}

Despite the robust computational development and wide usage of phylogenetic birth-death models for phylodynamic inference, their biological expressiveness is limited by the assumption that particles do not interact.
Interactions may be essential to evolutionary dynamics.
For example, environmental carrying capacity is a fundamental constraint on the long-term dynamics of any evolving population, and models of experimental microbial evolution generally allow for a transition from exponential growth to stationary phase as the population approaches capacity \citep{Baake2019}.
As another example, although the simple MTBDP facilitates modeling phenotypic selection via type-dependent birth and death rates, this does not capture \emph{frequency-dependent selection}, where the fitness of a given type depends on the distribution of types in the population.
In both of these examples, birth and death rates depend on the state of the population process, and this breaks the tree message-passing structure that phylogenetic birth-death models rely on.

As a motivating biological setting for the ideas to follow, 
we consider the somatic evolutionary process of \emph{affinity maturation} of antibodies 
in micro-anatomical structures called \emph{germinal centers} (GCs), 
which transiently form in lymph nodes during an adaptive immune response \cite[reviewed in][]{Victora2014-hm,Mesin2016-yu,Shlomchik2019-bo,victora2022germinal,Liu2023-wf}.
In a GC, B cells---the cells that make antibodies---diversify and compete based on the ability of the antibodies they express to recognize a foreign \emph{antigen} molecule.
As GC B cells proliferate, they undergo targeted mutations in the genomic locus encoding the antibody protein that can modify its antigen binding affinity (they undergo type transitions).
Via signaling from other GC cell types, the GC is able to monitor the binding phenotype of the B-cell population it contains, and provide survival signals to B cells with the highest-affinity antibodies (i.e., birth and death rates depend on type).

GCs have been studied extensively in mouse models that allow for experimental lineage tracing and manipulation of the B-cell population process.
In particular, B cells can be \emph{fate mapped} by genetically engineering them to express a fluorescent protein that marks them with a randomized color at the beginning of the GC evolutionary process \citep{Tas2016, Mesin2020, Pae2021-uu}.
These initially random colors are non-randomly inherited by descendant cells, so a sample of the GC B-cell population at a future time can be partitioned into \emph{lineages} of cells that share distinct common ancestors at the time of the initial color marking.
Phylogenetic inference can then be used to reconstruct the evolutionary history of a GC B-cell lineage using the DNA sequences of the sampled B cells \citep{DeWitt2018}.

GC B cells compete for limited proliferative signaling based on the antigen binding affinity of their B-cell receptors, and the population distribution of binding affinities generally improves as affinity maturation unfolds, so a given binding phenotype may be high-fitness early in the process, but low-fitness later when the population distribution of affinity has improved.
This invokes frequency-dependent selection, where the birth and death rates should depend on the population distribution of types.
GCs are observed to reach a steady-state carrying capacity of several thousand cells, based on limited cell-mediated proliferative signaling, so carrying capacity is likely also important, meaning that birth and death rates should depend on the total population size.

Phylodynamic models have the potential to reveal how evolutionary dynamics is orchestrated in GCs to shape antibody repertoires and immune memories.
However, phylogenetic birth-death models cannot accommodate key features of this system.
\citet{amitai2017population} presented a simulation study using a birth-death model with competition to investigate features of the GC population process, but such agent-based simulations are not amenable to likelihood-based inference for partially observed histories.
This motivates us to investigate a class of interacting MTBDPs that preserve tree-message passing for tractable likelihoods, and could thus be used in phylogenetic birth-death models.

\subsection{Mean-field interactions between replica birth-death processes}

Mean-field theories are a fundamental conceptual tool in the study of interacting particle systems.
The ideas originated in statistical physics and quantum mechanics as a technique to reduce many-body problems---in fluids, condensed matter, and disordered systems---to effective one-body problems \cite[see][]{Parisi2007, kadanoff2009more}.
The theory was extensively developed in the context of general classes of stochastic processes, and has since been widely applied across many scientific domains \cite[see][for a review of theory and applications]{MR4489768, MR4489769}.

Motivated by the setting of GC evolutionary dynamics described above, with population-level interaction among many fate-mapped lineages, we set out to develop a mean-field model that couples the birth and death rates in a focal MTBDP (with $d$ types) to the empirical distribution of states---i.e., the \emph{mean-field}---over an exchangeable system of $N$ replica MTBDPs.
More concretely, this empirical distribution process is a stochastic process taking values in the space of probability measures on $\mathbb{N}_0^\typ$, where $\N_0$ denotes the non-negative integers: the mass assigned by this measure-valued process at time $t \ge 0$ to a vector $\by = (y_1, y_2, \ldots, y_\typ) \in \mathbb{N}_0^\typ$ is the proportion of replica processes (including the focal process) that have $y_k$ individuals of type $k$ for $1 \le k \le \typ$.

We prove that the empirical distribution process of the $N$ replicas converges to a deterministic probability measure-valued flow as $N \to \infty$.
Using the \emph{propagation of chaos} theory \cite[see][for surveys of this vast area and references to its many applications]{MR4489768, MR4489769, Sznitman1991} we moreover show that in this limit, the replicas effectively decouple, and the focal process can instead be said to couple to a deterministic external field.
This external field is \emph{self-consistent} in the sense that, at any time $t \ge 0$, it is given by the very distribution of the state of the focal process.
We calculate self-consistent fields by solving limiting nonlinear forward equations for the focal process.
A key feature of this limit is that it restores message-passing likelihoods in the phylogenetic birth-death model setting, allowing for tractable phylodynamic models with interactions.

We note that there has been some work on mean-field models in the area of superprocesses (continuum analogs of branching processes) -- see \citet{overbeck1995superprocesses, Overbeck_96}. Finally, \citet{thai2015birth} is tangentially related to our work in that it treats a particular question concerning mean-field interacting single-type birth-death processes.
As the author of this paper observes regarding the literature about mean-field models and propagation of chaos,
``\textellipsis there are few results in discrete space.''

\subsection{Structure of the paper}
The rest of the paper is structured as follows. In Section \ref{sec:theory} we construct a system of MTBDPs that can model the properties and interactions between particles we have discussed so far. Moreover, this section contains our main theoretical results. Section~\ref{sec:proofs} is dedicated to their proofs. In Section~\ref{sec:examples} we analyze a special case of an MTBDP system numerically. Finally, in Section~\ref{sec:discussion}, we discuss our findings, and compare them to other relevant results from the literature.

\section{Theoretical results: a mean-field interacting multi-type birth-death process with general rates}
\label{sec:theory}

We start by describing a finite system of fairly general symmetrically interacting MTBDPs for which the interaction may be
locally strong but is 
globally weak in the sense that different MTBDPs interact only via the empirical distribution of their states.
Ultimately, we are interested in the joint law of a finite number of focal processes within an infinite system of such mean-field interacting MTBDPs.
To this end, we establish that the process of the empirical distribution of families converges to a deterministic probability measure-valued flow. 
Any finite number of MTBDPs become asymptotically independent and identically distributed.
In the limit, the law of any given focal process can be described by a time-inhomogeneous MTBDP.
The time inhomogeneity comes from the deterministic probability measure-valued background flow 
that also describes the one-dimensional marginal distributions of the focal process.

One main contribution of our analysis compared to previous studies is that we allow for a quite general transition rate structure. 
The rate of a single MTBDP is only restricted to be of at most linear growth 
and Lipschitz continuous.
To deal with the technical challenges that come with these general assumptions, 
we employ a localization technique and approximate the general system by one that has
bounded transition rates. 
A key feature is that the system with bounded rates is in a certain sense close to the one with unbounded rates, uniformly in the system size.

Consider a finite set of types $\{1,\ldots,\typ\}\eqqcolon[\typ]$. 
The application we have in mind is that each type represents a certain affinity of B-cell receptors. 
We equivalently refer to the cells as particles, 
in line with the terminology used in the branching process literature. 
At the outset, let's envision a germinal center that initially contains a finite collection of $N\in \N$ such B-cells. 
The progeny process of each of the~$N$ founding cells in this GC can be modeled as an MTBDP.
During the process of antibody affinity maturation, cells can divide into two daughters of the same type, mutate to one of the other~$(\typ-1)$ affinity types, or die, according to specified rates.
The interaction within lineages is (possibly) \emph{strong}, 
whereas the interaction between the $N$ lineages is \emph{weak}. 
This means that the rates for the $j$th lineage depend on its state (locally-strong)
and on the empirical distribution of MTBDP states over the $N$ lineages (globally-weak). 
Note that this includes the special case of rates that depend on the global empirical type distribution aggregated over all $N$ families in the GC.
Initially, there are $N\in \N$ founding particles within the GC.
A state of this system is then given by $\bz=(\bz_{1},\ldots,\bz_{N})\in (\N_0^{\typ})^N$, 
where for $j\in [N]$ and $i\in[\typ]$, 
$z_{j,i}$ counts the number of type-$i$ particles in the $j$th MTBDP.
Let $\Mp(\N_0^{\typ})$ denote the probability measures on $\N_0^{\typ}$.
This space is embedded into the Banach space of finite signed measures on $\N_0^{\typ}$ equipped with the total variation norm.
For $\nu\in \Mp(\N_0^{\typ})$ and $\by\in \N_0^{\typ}$, 
let $\nu_{\{\by\}}\coloneqq \nu(\{\by\})$. 
Then the total variation distance between $\nu,\nu'\in \Mp(\N_0^{\typ})$ is 
$\lVert \nu -\nu'\rVert_{TV} \coloneqq \frac{1}{2}\sum_{\by\in \N_0^{\typ}}\lvert \nu_{\{\by\}}-\nu_{\{\by\}}'\rvert.$ 

Let $\Mpn(\N_0^{\typ})\coloneqq\{\frac{1}{N} \sum_{j=1}^N \delta_{\bz_{j}}\in\Mp(\N_0^{\typ}):\, \bz\in (\N_0^{\typ})^N\}$, i.e. 
the probability measures that can arise as an empirical distribution of an MTBDP system with~$N$ initial particles. 
The \emph{empirical distribution of MTBDP states} of an $N$-system in state~$\bz$ is 
\[\pi_{\bz} \coloneqq \frac{1}{N} \sum_{j=1}^N \delta_{\bz_{j}}\in \Mpn(\N_0^{\typ}).\]
For example, for $\by=(y_1,\ldots,y_\typ)\in \N_0^{\typ}$ and $\bz\in (\N_0^{\typ})^N$, 
$\pi_{\bz}(\{\by\})$ counts the relative frequency of lineages with composition~$\by$, 
i.e.\ with $y_i$ particles of type~$i$, $i\in[\typ]$.

Every successive change in the system affects only one particle at a time with a rate depending on the local state of its lineage, 
and the empirical distribution over the population of $N$ lineages. 
That is, for $i \ne k\in [\typ]$, we have the following per lineage rates of various events
\begin{align*}
	b^i&:\N_0^{\typ}\times \Mp(\N_0^{\typ})\to \R_+\qquad (\textit{birth-rate of type-$i$ particles}),\\
	d^i&:\N_0^{\typ}\times\Mp(\N_0^{\typ})\to \R_+ \qquad (\textit{death-rate of type-$i$ particles}), \\
	m^{i,k}&:\N_0^{\typ}\times \Mp(\N_0^{\typ})\to \R_+ \qquad (\textit{mutation-rate from type-$i$ to type-$k$ particles}).
\end{align*}
Throughout we assume that if $z_{{j},i}=0$, 
then $b^i(\bz_{j},\pi_{\bz}) = d^i(\bz_{j},\pi_{\bz})=m^{i,k}(\bz_{j},\pi_{\bz})=0$ for all $k \in [D]$. We stress that the rates do not depend on $N$.

We will assume that the rates per lineage grow at most linearly with the number of particles in the lineage and that the rates are Lipschitz continuous in the following sense. 
For $\by\in \N_0^{\typ}$, set $\bulletize{\by} := \sum_{i \in [\typ]} y_i\in \N_0$.

\begin{assumption}\label{ass:growthlip}
	\begin{enumerate}[label=(\textbf{A.\arabic*})]
		\item \label{ass:growth} There exists a constant $L$ such that for all $i, k \in [\typ]$, $i\ne k$, $\by\in \N_0^{\typ}$ and $\nu\in \Mp(\N_0^{\typ})$,
		\begin{equation*}
			\begin{split}
				b^i(\by,\nu) \leq L (\bulletize{\by}+1),\quad 
				d^i(\by,\nu) \leq L (\bulletize{\by}+1),\quad 
				m^{i,k}(\by,\nu)\leq L (\bulletize{\by}+1).
			\end{split}
			\label{cond:lineargrowth}
		\end{equation*}
		\item \label{ass:lipschitz} There exists a constant $L$ such that for all $i, k \in [\typ]$, $i\ne k$, $\by,\by'\in \N_0^{\typ}$ and $\nu,\nu'\in \Mp(\N_0^{\typ})$,
		\begin{equation*}
			\begin{split}
				\lvert b^i(\by,\nu)-b^i(\by',\nu')\rvert &\leq L (\bulletize{|\by-\by'|}+\lVert \nu-\nu'\rVert_{TV}) ,\\
				\lvert d^i(\by,\nu) -d^i(\by',\nu')\rvert &\leq L (\bulletize{|\by-\by'|}+\lVert \nu-\nu'\rVert_{TV}),\\
				\lvert m^{i,k}(\by,\nu)-m^{i,k}(\by',\nu')\rvert &\leq L (\bulletize{|\by-\by'|}+\lVert \nu-\nu'\rVert_{TV}).
			\end{split}
			\label{cond:lipschitz}
		\end{equation*}
	\end{enumerate}
\end{assumption}

\begin{remark}
    To fully model features like carrying capacity constraints, an alternative would be to allow the rates to grow linearly with the mean of the measure, rather than bounding the contribution of the measure by a constant. In this case, the Lipschitz bounds would also depend on something like the Wasserstein-$1$ distances between the two measures involved. Proving similar results as ours under such assumptions is an open problem that we hope to return to in future work.
\end{remark}

The system of MTBDPs is formally described through its infinitesimal generator, which requires some notation. 
To add and remove particles of type~$i$ in the $j$th MTBDP,
we use $\be_{j,i}\in (\N_0^{\typ})^N$, 
where $(\be_{j,i})_{k,\ell}=\ind_{k}(i)\ind_{\ell}(j)$. 
The domain of the generator is described using specific function spaces.
We write $\bar{C}((\N_0^{\typ})^N)$ for the space of (continuous) bounded functions 
on $(\N_0^{\typ})^N$ and $\hat{C}((\N_0^{\typ})^N)$ for the space of 
(continuous) bounded functions on $(\N_0^{\typ})^N$ that vanish at infinity.
Moreover, we write $C_c((\N_0^{\typ})^N)$ for the space of compactly supported (finitely supported) (continuous) functions on $(\N_0^{\typ})^N$. For $\bz\in (\N_0^{\typ})^N$, set $\bulletized{\bz} \coloneqq \sum_{j\in [N]} \bulletize{(\bz_{j})} = \sum_{j \in [N]} \sum_{i \in [\typ]} z_{j,i} \in \N_0$. 

The generator $A^N$ of the finite system of interacting MTBDPs acts on $f \in \hat{C}((\N_0^{\typ})^N)$
via $A^Nf(\bz):=\sum_{j=1}^N (A_{\birth}^{N,j}+A_{\death}^{N,j}+A_{\mut}^{N,j})f(\bz)$ with
\begin{align*}
	A_\birth^{N,j}f(\bz)&:=\sum_{i=1}^{\typ} b^i(\bz_{j},\pi_{\bz})[f(\bz+\be_{j,i})-f(\bz) ]\\
	A_\death^{N,j}f(\bz)&:=\sum_{i=1}^{\typ} d^i(\bz_{j},\pi_{\bz})[f(\bz-\be_{j,i})-f(\bz) ]\\
	A_\mut^{N,j}f(\bz)&:=\sum_{i,k \in [\typ], \, i \ne k}  m^{i,k}(\bz_{j},\pi_{\bz})[f(\bz+\be_{j,k}-\be_{j,i})-f(\bz) ].
\end{align*}

Define \[\Delta_N\coloneqq\{f \in \hat{C}((\N_0^{\typ})^N):\ \bz \mapsto \bz_{\bullet \bullet} f(\bz) \in \bar{C}((\N_0^{\typ})^N),\ A^N f \in \hat{C}((\N_0^{\typ})^N)\}.\]

\begin{proposition}[Feller property for finite system] \label{prop:fellerfinite}

The closure of $\{(f, A^N f): f \in \Delta_N\}$ is single-valued and
generates a Feller semigroup on $\hat{C}((\N_0^{\typ})^N)$.
Moreover, $C_c((\N_0^{\typ})^N)$ is a core for this generator.
\end{proposition}
 
The proof of the proposition is in \S~\ref{sec:propertiesfinite}. 
Write
\[\bZ^N(t) \coloneqq (\bZ^{N}_1(t),\ldots,\bZ^{N}_N(t))
\] 
for a process with the semigroup guaranteed by Proposition~\ref{prop:fellerfinite}
and set $\bZ^N\coloneqq(\bZ^{N}(t))_{t\geq 0}$.

The system exhibits exchangeability among the MTBDPs due to the symmetries of the rates, 
provided that their initial distribution is also exchangeable.
To formally establish this property, we utilize the {\em Markov mapping theorem}. 
As a result of this analysis, we also derive the Markovian nature of the system's empirical distribution process, subject to suitable initial conditions.

The empirical distribution of states in the system at time~$t \ge 0$ is 
\[\Pi^N(t) \coloneqq\frac{1}{N}\sum_{j=1}^N \delta_{\bZ_j^N(t)}.\]
The $\Mpn(\N_0^{\typ})$-valued empirical measure process is $\Pi^N=
(\Pi^N(t))_{t\geq 0}$.

Its infinitesimal generator $B^N$ acts on a subset of $\bar{C}(\Mpn(\N_0^{\typ}))$, 
the bounded continuous functions on $\Mpn(\N_0^{\typ})$.  
More specifically, 
$B^N$ acts on functions of the form
\[h(\nu)=\frac{1}{N!} \prod_{\substack{\by \in \N_0^{\typ}:\\  \by \in \supp(\nu)}}(N\nu(\{\by\}))!\sum_{\bz\in (\N_0^{\typ})^N: \pi_{\bz}=\nu} f(\bz),\] with $f\in \Delta_N$, 
via $B^Nh(\nu):=(B_\birth^{N}+B_\death^{N}+B_\mut^{N})h(\nu)$, where 
\begin{align*}
	B_\birth^{N}h(\nu)&:=\sum_{\by \in \N_0^{\typ}}\sum_{i=1}^{\typ} N \nu_{\{\by\}} b^i(\by ,\nu)\Big[h\Big(\nu+\frac{\delta_{\by +\be_i}-\delta_\by }{N} \Big)-h(\nu)\Big] , \\
	B_\death^{N}h(\nu)&:=\sum_{\by \in \N_0^{\typ}}\sum_{i=1}^{\typ} N\nu_{\{\by\}}  d^i(\by ,\nu)\Big[h\Big(\nu+\frac{\delta_{\by -\be_i}-\delta_\by }{N} \Big)-h(\nu)\Big], \\
	B_\mut^{N}h(\nu)&:=\sum_{\by \in \N_0^{\typ}}\sum_{i,k \in [\typ], \, i \ne k} N \nu_{\{\by\}} m^{i,k}(\by ,\nu)\Big[h\Big(\nu+\frac{\delta_{\by +\be_k-\be_i}-\delta_\by }{N} \Big)-h(\nu)\Big],
\end{align*}
with $\be_i$ the $i$th unit vector in $\N_0^{\typ}$.

To formally state the exchangeability of the system and the Markovian nature of~$\Pi^N$, we require some notation. Let $\alpha^N(\nu,\dd \bz)$ be a kernel from $\Mpn(\N_0^{\typ})$ to $(\N_0^{\typ})^N$ 
defined via
\[\alpha^N(\nu, \dd \bz)\coloneqq\frac{1}{N!} \prod_{\substack{\by \in \N_0^{\typ}:\\  \by \in \supp(\nu)}}(N\nu(\{\by\}))!  \sum_{\bx\in (\N_0^{\typ})^N:\, \pi_{\bx}=\nu}\delta_{\bx}(\bz),\]
i.e.\ $\alpha^N(\nu,\cdot)$ puts mass uniformly among all the system states $\bx \in (\N_0^{\typ})^N$ 
that are compatible with an empirical distribution~$\nu$.
For $f\in \bar{C}((\N_0^{\typ})^N)$, 
we write $\alpha^Nf(\cdot)=\sum_{\bz\in (\N_0^{\typ})^N} f(\bz)\alpha^N(\cdot,\dd \bz)$.
(In particular, $\alpha^Nf\in \bar{C}(\Mpn(\N_0^{\typ}))$.) 

\begin{proposition}[Exchangeability]\label{prop:exchangeability}
	Let $\nu^N \in \Mpn(\N_0^{\typ})$ and assume $\bZ^N(0)$ has distribution $\alpha^N(\nu^N,\cdot)$. 
	For all $t\geq 0$, $\bZ^N(t)=(\bZ^{N}_1(t),\ldots,\bZ^{N}_N(t))$ is exchangeable and 
	$\Pi^N$ is a Markov process with generator $B^N$.
\end{proposition}

In what follows, we consider the limit of large systems.
For the germinal center application, this means that we assume the initial number of B-cells to be large.
Any dependence of the rates on the \emph{total mass} therefore is meant to be relative to the initial mass.

Our first main result describes the behavior of $\Pi^N$ in the limit of large systems.  We adopt the usual notation that if $I$ is a closed subinterval of $\R_+$ and $E$ is a metric space, then $D(I, E)$ is the Skorohod space of right-continuous, left-limited functions from $I$ to $E$.

\begin{theorem}[Convergence of empirical measure process] \label{thm:conv}
	Assume $\bZ^N(0)$ has distribution $\alpha^N(\nu^N,\cdot)$ for $\nu^N\in \Mpn(\N_0^{\typ})$ satisfying $\nu^N\xrightarrow{N\to\infty} \nu\in \Mp(\N_0^{\typ})$. 
	Then there exists a unique solution to the initial value problem: $v(0)=\nu$ and for all $\by\in \N_0^{\typ}$,
	\begin{equation}
		\label{eq:limitode}
		\begin{split}
			v_{\{\by\}}'(t)=&-v_{\{\by\}}(t) \sum_{i=1}^{\typ} \big(b^i(\by,v(t))+d^i(\by,v(t))+\sum_{k=1,\, k\neq i}^{\typ} m^{i,k}(\by,v(t)) \big)\\
			&+\sum_{i=1}^{\typ} \big(v_{\{\by-\be_i\}}(t)b^i(\by-\be_i,v(t)) + v_{\{\by+\be_i\}}(t) d^i(\by+\be_i,v(t)) \\
			&\qquad \ + \sum_{k=1,\, k\neq i}^{\typ} v_{\{\by-\be_k+\be_i\}}(t)\, m^{i,k}(\by-\be_k+\be_i,v(t))\big).
		\end{split}
	\end{equation}
	(with the convention that for $\by\notin \N_0^{\typ}$, $v_{\{\by\}}(t)=b^i(\by,v(t))=d^i(\by,v(t))=m^{i,k}(\by,v(t))=0$).
	Moreover, 
	\begin{equation}
		\Pi^N\xRightarrow{N\to\infty}v
	\end{equation}
(that is, the sequence of $D(\R_+, \Mpn(\N_0^{\typ}))$-valued random elements $(\Pi^N)_{N \in \N}$ converges in distribution to the deterministic (continuous) function $v$).
\end{theorem}

\begin{remark}
\label{rem:chaos_implication}
From Theorem~\ref{thm:conv}, the continuity of $v$, and the continuous mapping theorem, it follows that,
under the conditions of Theorem~\ref{thm:conv},
\[
\Pi^N(t)\xRightarrow{N\to\infty}v(t)
\]
for every $t \ge 0$ (cf. Theorem 23.9 of \cite{Kallenberg2021}).  Hence, under the assumptions of Theorem~\ref{thm:conv},
it follows from the theory of propagation of chaos, see Proposition~2.2 of \cite{Sznitman1991}, that for every $k \in \N$ and $t \ge 0$ the $(\N_0^\typ)^k$-valued random vector $(\bZ_1^N(t), \ldots, \bZ_k^N(t))$
converges in distribution and that the limiting distribution is $v(t)^{\otimes k}$.  We can do better than this, as the following second major result shows.  
\end{remark}

\begin{corollary}[Convergence of focal processes]
\label{cor:conv}
Under the conditions of Theorem~\ref{thm:conv}, for each $k \in \N$ there is convergence in distribution of the $D(\R_+,(\N_0^\typ)^k)$-valued sequence of random elements $\{(\bZ_1^N, \ldots, \bZ_k^N)\}_{N \in \N}$ to $(\bZ_1^\infty, \ldots, \bZ_k^\infty)$, where $\bZ_1^\infty, \ldots, \bZ_k^\infty$ are i.i.d. time-inhomogeneous MTBDPs with common initial distribution $\nu$ and the  birth, death, and mutation rates of $\bZ_j^\infty$, $1 \le j \le k$, at time~$t\ge 0$ are given by $b^i(\bZ_j(t),v(t))$, $d^i(\bZ_j(t),v(t))$, and $m^{i,\ell}(\bZ_j(t),v(t))$ for $i,\ell\in[\typ]$, $i\ne \ell$.  
\end{corollary}

The remainder of this paper consists of two sections. The next section contains the proofs of our theoretical results from Section 2. In Section 4 we provide a numerical example of an MTBDP where the rates depend on the number of particles of each type that are present in the system.

\section{Proofs of the main results} \label{sec:proofs}

We begin by proving the properties of the finite system and its empirical measure process. Next, we examine a system of independently evolving Yule-type processes that dominates the interacting MTBDP system, deriving results essential for establishing the properties of $\Pi^N$. In the third section, we prove Theorem~\ref{thm:conv} using the properties of a localized system. Finally, we establish the properties of the localized system, which we utilize in the proof of Theorem~\ref{thm:conv}.

\subsection{Properties of the finite system} \label{sec:propertiesfinite}

We initiate our analysis by proving the result concerning the generator of the finite system of MTBDPs.
\begin{proof}[Proof of Proposition~\ref{prop:fellerfinite}]
	The proof consists of checking the
	conditions of Theorem~3.1 in Chapter~8 of \cite{Ethier1986}.
	The kernel that plays the role of the kernel
	$x \mapsto \lambda(x) \mu(x, \dd y)$ in \cite{Ethier1986} is here
	\begin{equation}
		\begin{split}
			\bz \mapsto
			\sum_{j=1}^N \sum_{i=1}^{\typ}
			\biggl[
			b^i(\bz_{j},\pi_{\bz})\delta_{\bz+\be_{j,i}} + d^i(\bz_{j},\pi_{\bz})\delta_{\bz-\be_{j,i}} +\sum_{k \in [\typ], \, k\ne i}  \hspace{-4mm} m^{i,k} (\bz_j,\pi_{\bz}) 
			\delta_{\bz+\be_{j,k}-\be_{j,i}} \biggr] \\ 
		\end{split}
		\label{def:jump_kernel}
	\end{equation}
	We will take the functions $\gamma$ and $\eta$ that appear in the statement
	of that result to both be $\bz \mapsto \chi(\bz) \coloneqq (\bz_{\bullet \bullet} \vee 1)$.
	
	First note that $\bz \mapsto 1/\chi(\bz) \in \hat{C}((\N_0^{\typ})^N)$, as required in \cite{Ethier1986}.
	Secondly,
	\begin{equation}
		\sup_{\bz \in (\N_0^{\typ})^N} \sum_{j=1}^N\sum_{i=1}^{\typ}
		\biggl[ b^i(\bz_{j},\pi_{\bz})
		+d^i(\bz_{j},\pi_{\bz})
		+\sum_{k \in [\typ], \, k \ne i}  m^{i,k}(\bz_{j},\pi_{\bz}) 
		\biggr] \bigg / \chi(\bz) < \infty
		\label{bd:lambda_pi_ratio}
	\end{equation}
	by \ref{ass:growth},
	and so hypothesis (3.2) of \cite{Ethier1986} is satisfied.
	
	If $\bz'$ is a point in the support of the measure on the right-hand size of \eqref{def:jump_kernel}, 
	then $\lvert\bulletized{\bz} -\bulletized{\bz'}\rvert \leq 1$ 
	and hence hypothesis (3.3) of
	\cite{Ethier1986} is satisfied.
	
	Combining the bound \eqref{bd:lambda_pi_ratio} with the observation
	of the previous paragraph shows that hypotheses (3.4) and (3.5) of \cite{Ethier1986} hold, and this completes the proof.
\end{proof}

Next, we establish the exchangeability of the finite system and demonstrate the Markovianity of the empirical measure process.

\begin{proof}[Proof of Proposition~\ref{prop:exchangeability}]
	We first want to apply \citet[Corollary~3.5]{Kurtz1998}. 
	Note that for any $h\in \bar{C}(\Mpn(\N_0^{\typ}))$ and $\pi_{\bz}\in \Mpn(\N_0^{\typ})$, we have $\int h(\pi_{\by})\alpha^N(\pi_{\bz},\dd \by)=h(\pi_{\bz}).$ 
	Define 
	\begin{equation}
		C^N=\left \{ \left ( \textstyle\sum_{\by\in (\N_0^{\typ})^N} f(\by) \alpha^N(\cdot, \dd \by ),\sum_{\by\in (\N_0^{\typ})^N} A^Nf(\by) \alpha^N(\cdot, \dd \by ) \right ):f\in \Delta_N \right \}. \label{eq:generatorempiricalN}
	\end{equation}
	We have to verify (the somewhat technical condition) that each solution of the extended forward equation of $A^N$ 
	corresponds to a solution of the martingale problem.
	Assume for now this is true.
	We show in Lemma~\ref{lem:generatoridentity} below that for $f(\bz)=\prod_{j=1}^N g_j(\bz_{j})$ with $g_j\in \hat{C}(\N_0^{\typ})$, $\alpha^N (A^Nf)(\pi_{\bz} )
	= B^N(\alpha^N f)(\pi_{\bz}).$ 
	In particular, $\Pi^N$ solves the $C^N$ martingale problem.
	Thus, by Corollary~3.5 of \citet{Kurtz1998} (with $\gamma(\bz)=\pi_{\bz}$ ), 
	$\Pi^N$ is a Markov process.
	Moreover, by Theorem~4.1 of \citet{Kurtz1998}, $\bZ^N(t)$ is then exchangeable.
	
	It remains to verify that each solution of an extended forward equation of $A^N$ corresponds to a solution of the martingale problem. By Lemma~3.1 of~\citet{Kurtz1998}, 
	it is enough to verify that $A^N$ satisfies the conditions of Theorem~2.6 of \citet{Kurtz1998}, that is, 
	that $A^N$ is a pre-generator and 
	Hypothesis~2.4 of~\citet{Kurtz1998} is satisfied. 
	Because $(\N_0^{\typ})^N$ is locally compact, 
	the latter is satisfied by Remark~2.5 of~\citet{Kurtz1998}. 
	Another consequence of local compactness is that 
	for $A^N$ to be a pre-generator,
	it is enough to verify that $A^N$ satisfies the positive maximum principle \cite[p.4]{Kurtz1998}, 
	which is easily seen to be the case.
\end{proof}

The following lemma is a technical result used in the proof of Proposition~\ref{prop:exchangeability}.
\begin{lemma}\label{lem:generatoridentity}
	For $j\in [N]$, let $g_j\in \hat{C}(\N_0^{\typ})$ and set $f(\bz)= \prod_{j=1}^Ng_j(\bz^{N}_j)$. 
    Then, for $\bz\in (\N_0^{\typ})^N$ and $\pi_{\bz}\in \Mpn(\N_0^{\typ})$, we have
	\begin{equation*}
		\alpha_N (A^Nf)(\pi_{\bz} )
		= B^N(\alpha_N f)(\pi_{\bz}).
	\end{equation*}
\end{lemma}
\begin{proof}
	We will only show $\alpha_N (A^N_\birth f)(\pi_{\bz})=B^N_\birth(\alpha_Nf)(\pi_{\bz})$. That $\alpha_NA^N_\death f(\pi_{\bz})=B^N_\death(\alpha_Nf)(\pi_{\bz})$ and $\alpha_NA^N_\mut f(\pi_{\bz})=B^N_\mut(\alpha_Nf)(\pi_{\bz})$ can be proven in a similar vein. The result then follows from the linearity of $A^N$ and $B^N$.
	We have 
	\begin{align*}
		&\alpha_N A^N_\birth f(\pi_{\bz} )\\
        &\ =  \frac{1}{N!} \prod_{\substack{\by' \in \N_0^{\typ}:\\  \by \in \supp(\pi_{\bz})}}\hspace{-.4cm}(N\pi_{\bz}(\{\by'\}))! \sum_{\substack{\bx\in (\N_0^{\typ})^N:\\ \pi_{\bx}=\pi_{\bz}}} \sum_{j=1}^N \sum_{i=1}^{\typ}  b^i(\bx^{j},\pi_{\bz}) [f(\bx+\be_{j,i})-f(\bx)] \\
		&\ = \frac{1}{N!} \prod_{\substack{\by' \in \N_0^{\typ}:\\  \by \in \supp(\pi_{\bz})}}\hspace{-.4cm}(N\pi_{\bz}(\{\by'\}))! \sum_{\by \in \N_0^{\typ}}\sum_{i=1}^{\typ}   \sum_{j=1}^N \sum_{\substack{\bx\in (\N_0^{\typ})^N:\\ \pi_{\bx}=\pi_{\bz}}} \ind_{\by}(\bx^{j})  b^i(\bx^{j},\pi_{\bz}) [f(\bx+\be_{j,i})-f(\bx)]\\
		&\ = \sum_{\by\in \N_0^{\typ}} \sum_{i=1}^{\typ} N\pi_{\bz}(\by) b^i(\by,\pi_{\bz}) \Big[(\alpha_Nf)\big(\pi_{\bz}+(\delta_{\by+\be_i}-\delta_\by)/N\big)-(\alpha_Nf)(\pi_{\bz}) \Big]\\
		&\ =B^N_\birth(\alpha_Nf)(\pi_{\bz}).
	\end{align*}
\end{proof}

\bigskip\noindent
{\bf For the remainder of this section we will assume that the conditions of Theorem~\ref{thm:conv} hold; that is, \ref{ass:growth} and \ref{ass:lipschitz} hold,
and the sequence
$\nu^N\in \Mpn(\N_0^{\typ})$, $N \in \N$, satisfies $\nu^N\xrightarrow{N\to\infty} \nu\in \Mp(\N_0^{\typ})$.}

\bigskip

\subsection{A dominating system of Yule-type processes}\label{sec:dominatingyule}

It will be useful to compare $\bZ^N$ to a system of asymptotically independent multi-type pure-birth-like processes that will have simultaneous births of different types.
Even though these Markov processes are $\N_0^{\typ}$-valued, 
they inherit several useful properties from the classic Yule process.
Let $\bR^N=(\bR_1^N, \ldots, \bR_N^N)$ be the $(\N_0^{\typ})^N$-valued Markov process 
transitioning from $(\N_0^{\typ})^N \ni 
(\br_1, \ldots, \br_N) \to (\br_1, \ldots, \br_N) +
(\mathbf{0}, \ldots, \mathbf{0}, \mathbf{1}, \mathbf{0}, \ldots, \mathbf{0})$
at rate $6L \typ^2 \bulletize{(\br_j)}$,
where $\mathbf{0} \in \N_0^{\typ}$ is the vector of all $0$s 
and $\mathbf{1} \in \N_0^{\typ}$ is the vector of all $1$s.   Define $\rho:\N_0^\typ \to \N_0^\typ$  by
$\rho(\br):= \bulletize{\br} \mathbf{1}$ (that is, if we think of $\br$ as a collection of particles of different types, then $\rho(\br)$ replaces each particle by $\typ$ particles where there is one particle of every one of the $\typ$ types.  Suppose that
$\bR^N(0)$ has distribution $\alpha^N(\nu^N, \cdot) \circ (\rho, \ldots, \rho)^{-1}$. 
It follows that each $\bulletize{(\bR_j^N)}$ is an autonomous Markov process on $\typ \N_0$ that transitions from
$\typ r$ to $\typ r + \typ$ at rate $6L\typ^3 r$.  Consequently, $\typ^{-1}  \bulletize{(\bR_j^N)}$ is
a Yule process that transitions from state $x$ at rate $6L\typ^3 x$ (that is, the split rate per particle is $6L\typ^3$).

For $t \ge 0$ define $\Pi^N_{\bR}(t):=\frac{1}{N}\sum_{j=1}^N\delta_{\bR_j^N(t)}$ and set $\Pi^N_{\bR}:=(\Pi^N_{\bR}(t))_{t\geq 0}$.

\begin{lemma}
\label{lem:convpir}
\begin{itemize}
    \item[i)]
    For each $k \in \N$, there exist $\bR^\infty_j$, $j \in [k]$, such that the sequence $\{(\bR^N_1, \ldots, \bR_k^N)\}_{N \in \N}$ 
    converges in distribution to $(\bR^\infty_1, \ldots, \bR^\infty_k)$.  The $\bR^\infty_j$, $j \in [k]$ are i.i.d. Markov processes. Each one has initial distribution $\nu \circ \rho^{-1}$ and the same transition dynamics as the $\bR_\ell^N$, $\ell \in [N]$, $N \in \N$, have in common.
    \item[ii)]
    There is a unique solution $r\coloneqq(r(t))_{t\geq 0}$ to the initial value problem: 
$r(0)=\nu \circ \rho^{-1}$ and for all $\by\in \N_0^{\typ}$,
\begin{equation}
	\label{eq:limitodeyule}
	\begin{split}
		r_{\{\by\}}'(t)=&-6L \typ^2 \bulletize{\by} r_{\{\by\}}(t) +6L \typ^2 \bulletize{(\by-\mathbf{1})} r_{\{\by-\mathbf{1}\}}(t),
	\end{split}
\end{equation}
where $r_{\{\by\}}(t)=0$ for $\by\notin \N_0^{\typ}$.
\item[iii)]
We may build $(\bR^N)_{N \in \N}$ on a suitable probability space so that 
\[
\Pi^N_{\bR} \xrightarrow{N \to \infty} r, \quad \text{a.s.}
\]
\end{itemize}

\end{lemma}

\begin{proof}
i) Recall that the distribution of $\bR^N(0)$
is $\alpha^N(\nu_0^N, \cdot) \circ (\rho, \ldots, \rho)^{-1}$.  It suffices to show that the projection
of this exchangeable probability measure onto the first $k$ coordinates of $(\N_0^{\typ})^N$ converges weakly to
the product probability measure
$(\nu \circ \rho^{-1})^{\otimes k}$ as $N \to \infty$.
Moreover, from Proposition~2.2 of 
\cite{Sznitman1991} it suffices to check that the sequence of probability measures on $\Mp(\N_0^{\typ})$ given by
$(\alpha^N(\nu^N, \cdot) \circ (\rho, \ldots, \rho)^{-1})\circ \pi^{-1}$, $N \in \N$, converges weakly to the unit point mass at the probability measure $\nu \circ \rho^{-1}$ (recall that $\bz \mapsto \pi_\bz$ is the map that takes $\bz \in (\N_0^\typ)^N$ to $\frac{1}{N}\sum_{j = 1}^N \delta_{\bz_j} \in \Mpn(\N_0^\typ)$). However, it is clear by construction that $(\alpha^N(\nu^N, \cdot) \circ (\rho, \ldots, \rho)^{-1})\circ \pi^{-1}$ is simply the unit point mass at the probability measure $\nu^N \circ \rho^{-1}$.

\bigskip
\noindent
ii) Note that \eqref{eq:limitodeyule} is just the Kolmogorov forward equations for a Markov process with transition dynamics the common transition dynamics of
$\bR_j^N$, $j \in [N]$, $N \in \N$, and initial distribution $\nu \circ \rho^{-1}$; that is, for a Markov process with the common distribution of $\bR_j^\infty$, $N \in \N$.  As we have remarked, such a Markov process is essentially a Yule process, and hence the Kolmogorov forward equations have a unique solution.

\bigskip
\noindent
iii) From (i) and Proposition~2.2 of \cite{Sznitman1991} we have that the empirical measures on the path space $D(R_+, \Mp(\N_0^{\typ}))$ given by
$\Sigma^N := \frac{1}{N} \sum_{j=1}^N \delta_{\bR_j^N}$ converge in distribution to the point mass at the common distribution of the $\bR^\infty_j$, $j \in \N$.
By Skorohod's coupling, see Theorem~5.31 in \cite{Kallenberg2021}, it is possible to build random variables with the distributions of the $\Sigma^N$, $N \in \N$, on a suitable probability space so that
$\Sigma^N$ converges almost surely to the point mass at the common distribution of the $\bR^\infty_j$, $j \in \N$.  We may, of course, also assume that $\bR^\infty_j$, $j \in \N$, are built on this probability space.

Fix $T > 0$. Let $\mathfrak{D}$ be a countable dense set in $[0,T]$ containing $\{0,T\}$.
By the continuous mapping theorem, with probability one we have that for all $\mathbf{m} \in \N_0^{\typ}$
\begin{equation}
\label{eq:convergence_on_D}
\Pi^N_{\bR}(t) (\{\by: y_i \geq m_i, \, i \in [\typ] \}) 
\xrightarrow{N\to\infty}
 r(t)(\{\by: y_i \geq m_i, \, i \in [\typ] \})
\end{equation}
for all $t \in \mathfrak{D}$.  By well-known results in real analysis, the monotonicity of the functions involved in the convergence in \eqref{eq:convergence_on_D}, plus the continuity of the right-hand side give firstly that the convergence holds for all $t \in [0,T]$ and secondly that the convergence is uniform.
Consequently, almost surely $\Pi^N_{\bR}(t)(\{\by\})$ converges uniformly to $r_{\{\by\}}(t)$ on $[0,T]$ for every $\by \in \N_0^{\typ}$.  

Given any $\epsilon > 0$ we can choose $K$ such that 
$\Pi^N_{\bR}(T) (\{\by: \bulletize{\by} > K \}) \le \epsilon$ for all $N$ and 
$r(T) (\{\by: \bulletize{\by} > K \}) \le \epsilon$.  Therefore, using the monotonicity of 
$\Pi^N_{\bR}(t) (\{\by: \bulletize{\by} > K \})$ and $r(t) (\{\by: \bulletize{\by} > K \})$
we have
\[
\limsup_{N \to \infty} \sup_{t \in [0,T]} \|\Pi^N_{\bR}(t) - r(t)\|_{\mathrm{TV}} \le 2 \epsilon.
\]
Since $T$ and $\epsilon$ are arbitrary, this completes the proof.
\end{proof}

For $\bz,\bz'\in(\N_0^{\typ})^N$, we write $\bz\leq \bz'$ if $z_{j,i}\leq z_{j,i}'$ for all $j,i$.

\begin{lemma}[Dominating pure-birth process coupling] \label{lem:dominationbirth}
	For each $N\in \N$
	we can couple $\bZ^N$ and $\bR^N$ together so that almost surely
	$\bZ^N(0)\leq \bR^N(0)$ and
    almost surely 
	for all $t \ge 0$, $j \in [N]$, and $i \in [\typ]$,
	$\lvert Z_{j,i}^N(t) -  Z_{j,i}^N(t-)\rvert \leq R_{j,i}^N(t) -  R_{j,i}^N(t-)$.  In particular, almost surely for all $t \ge 0$, $\bZ^N(t) \le \bR^N(t)$ and almost surely for all $0 \le s < t$, $j \in [N]$, and $i \in [\typ]$, $\lvert Z_{j,i}^N(t) -  Z_{j,i}^N(s)\rvert \leq R_{j,i}^N(t) -  R_{j,i}^N(s)$.  
\end{lemma}

\begin{proof}
    First note that, because $(\bR_1^N(0), \ldots, \bR_N^N(0))$ has the same distribution as $(\rho(\bZ_1^N(0)), \ldots, \rho(\bZ_N^N(0)))$, it is certainly possible to couple $\bZ^N(0)$ and $\bR^N(0)$ together in the prescribed manner.      
    
    Next observe that the rate at which a given $\bZ^N_j$ transitions to another state if $\bZ_j^N$ is in state $\bz_j \ne 0$ can be upper bounded using~\ref{ass:growth}:
        \begin{equation*}
            \begin{split}
                &\sum_{i=1}^\typ \bigg( b^i(\bz_j,\pi_{\bz})+d^i(\bz_j,\pi_{\bz})+\sum_{k\in [\typ], k\neq i} m^{i,k}(\bz_j,\pi_{\bz}) \bigg)\\
                 &\leq 3 L \typ^2 \big(1+\bulletize{(\bz_j)}\big)\leq 6L\typ^2 \bulletize{(\bz_j)}.
            \end{split}
        \end{equation*}
    The same inequality holds trivially when $\bz_j = 0$ by our assumption that in this case $b^i(\bz_{j},\pi_{\bz}) = d^i(\bz_{j},\pi_{\bz}) = m^{i,k}(\bz_{j},\pi_{\bz}) = 0$ for $i,k \in [\typ], \, i \ne k$.

	Thus, we can couple $\bZ^N$ to $\bR^N$ by restricting the possible jump times of $\bZ^N$ to the jump times of $\bR^N$
    and for  $\tau$  a jump time of $\bR^N$ such that the jump occurs in $\bR_j^N$ for some $j\in[\typ]$, 
    setting
	\[\bZ^N(\tau)=\begin{cases}
		\bZ^N(\tau-)+\be_{j,i}, &\text{w.p. } \frac{b^i(\bZ_j^N(\tau-),\Pi^N(\tau-))}{6L \typ^2 \bulletize{(\bR_j^N(\tau-))}},\qquad i \in[\typ]\\
		\bZ^N(\tau-)-\be_{j,i}, &\text{w.p. } \frac{d^i(\bZ_j^N(\tau-),\Pi^N(\tau-))}{6L \typ^2 \bulletize{(\bR_j^N(\tau-))}},\qquad i \in[\typ]\\
		\bZ^N(\tau-)+\be_{j,k}-\be_{j,i}, &\text{w.p. } 
        \frac{m^{i,k}(\bZ_j^N(\tau-),\Pi^N(\tau-))}{6L \typ^2 \bulletize{(\bR_j^N(\tau-))}},\quad i,k\in[\typ],\ i\ne k,\\
		\bZ^N(\tau-), &\text{otherwise}.
	\end{cases}\]
	It is straightforward to check that then $\bZ^N$ has the correct distribution
	and the other properties we want.
\end{proof}

{\bf From now on we assume $\bZ^N$ is constructed on the basis of the coupling in Lemma~\ref{lem:dominationbirth}.}

\bigskip

For two probability measures $\nu,\nu'\in \Mp(\N_0^{\typ})$, we say that $\nu'$ {\em stochastically dominates} $\nu$ if for every $\mathbf{m}\in \N_0^{\typ}$, $\nu(\{\by: y_i \geq m_i, \, i \in [\typ] \})\leq \nu'(\{\by: y_i \geq m_i, \, i \in [\typ] \})$; we then write $\nu \preceq \nu'$.

\begin{remark}The upper bound of $\bZ^N$ in terms of $\bR^N$ can be translated to a bound for their respective empirical measure processes.
To this end,
define $\Pi^N_{\bR}(t):=\frac{1}{N}\sum_{j=1}^N\delta_{\bR_j(t)}$ and set $\Pi^N_{\bR}:=(\Pi^N_{\bR}(t))_{t\geq 0}$. 
Because of Lemma~\ref{lem:dominationbirth}, 
we have $\Pi^N(t) \preceq \Pi^N_{\bR}(t)$ for every $t\geq 0$.
Moreover,
 $\Pi^N_{\bR}$ is non-decreasing, i.e.
for all $0 \le s < t$,
$\Pi^N_{\bR}(s) \preceq \Pi^N_{\bR}(t)$.
\end{remark}

\subsection{Proving convergence via localization} \label{sec:localizationproofs}

We employ a localization argument to establish Theorem~\ref{thm:conv}. 
The core concept involves freezing families that reach a certain size $\kappa\in \N$. 
By utilizing classic methods, we can prove the convergence of the empirical distribution for a system undergoing such freezing.

Let $\bZ^{N,\kappa}:=(\bZ^{N,\kappa}_1,\ldots,\bZ^{N,\kappa}_N)$ be the system of interacting MTBDPs
that is coupled to $\bZ^N$ by freezing lineages once they reach a state $\by$ where $\bulletize{\by}=\kappa$.
Notably, the construction of $\bZ^{N,\kappa}$ can therefore also be based on the system $\bR^N$ in the manner of Lemma~\ref{lem:dominationbirth}. 
Importantly,  $\bZ^{N,\kappa}(t)\leq \bR^N(t)$ holds for all $t\geq 0$.

Let $b^{i,\kappa}$, $d^{i,\kappa}$, and $m^{i,k,\kappa}$ be the
birth, death, and mutation rates of $\bZ^{N,\kappa}$.  For example,
$b^{i,\kappa}(\bz,\nu) = b^i(\bz,\nu) \ind(\bulletize{\bz} < \kappa)$
We may think of $\bZ^{N,\kappa}$ as a Markov process on the finite state space $(\kapemp)^N$ where $\kapemp := \{\by \in \N_0^\typ: \bulletize{\by} \le \kappa\}$.

The generator $A^{N,\kappa}$ of $\bZ^{N,\kappa}$ 
is then $A^{N,\kappa}f(\bz):=\sum_{j=1}^N (A_{\birth}^{N,j,\kappa}+A_{\death}^{N,j,\kappa}+A_{\mut}^{N,j,\kappa})f(\bz)$ for $f\in\hat{C}((\N_0^{\typ})^N)$ with
\begin{align*}
	A_\birth^{N,j,\kappa}f(\bz)&:=\sum_{i=1}^{\typ} b^{i,\kappa}(\bz_{j},\pi_{\bz})[f(\bz+\be_{j,i})-f(\bz) ]\\
	A_\death^{N,j,\kappa}f(\bz)&:=\sum_{i=1}^{\typ} d^{i,\kappa}(\bz_{j},\pi_{\bz})[f(\bz-\be_{j,i})-f(\bz) ]\\
	A_\mut^{N,j,\kappa}f(\bz)&:=\sum_{i,k \in [\typ], \, i \ne k}  m^{i,k,\kappa}(\bz_{j},\pi_{\bz})[f(\bz+\be_{j,k}-\be_{j,i})-f(\bz) ].
\end{align*}

Due to the state space of $\bZ^{N,\kappa}$ being (essentially) finite, rendering it compact, we can now state the following proposition (see \cite[Ch. 8.3.1]{Ethier1986}).

\begin{proposition}[Feller property for the system of frozen processes] \label{prop:fellerfinitetruncated}
The closure of $\{(f, A^{N,\kappa} f): f \in C((\N_0^{\typ})^N)\}$ is single-valued and
generates a Feller semigroup on $C((\N_0^{\typ})^N)$.
\end{proposition}

Also, in the system with frozen dynamics, the empirical distribution process is Markov.
To be precise,
define for $t\geq 0$,
$\Pi^{N,\kappa}(t):=\frac{1}{N}\sum_{j=1}^N\delta_{\bZ^{N,\kappa}_j(t)}\in \Mpn(\N_0^{\typ})$ 
and set $\Pi^{N,\kappa}:=(\Pi^{N,\kappa}(t))_{t\geq 0}$. 
Its infinitesimal generator
$B^{N,\kappa}$ is defined in the same way as $B^N$, 
but with the $\kappa$-frozen transition rates
and modified domain 
(because the domain of $A^{N,\kappa}$ is different). The following holds via Proposition \ref{prop:exchangeability},  since the rate functions of the frozen process satisfy~\eqref{ass:growthlip}.
\begin{remark}[Exchangeability]\label{rem:exchangeabilitytruncated}
	Let $\nu^N \in \Mpn(\N_0^{\typ})$ and assume $\bZ^{N, \kappa} (0)$ has distribution $\alpha^N(\nu^N,\cdot)$.  For all $t\geq 0$, $\bZ^{N,\kappa}(t)=(\bZ^{N,\kappa}_1(t),\ldots,\bZ^{N,\kappa}_N(t))$ is exchangeable and 
	$\Pi^{N,\kappa}$ is a Markov process with generator $B^{N,\kappa}$.
\end{remark}

The proof of Theorem~\ref{thm:conv} revolves around three key propositions, 
all of which will be proved in \S~\ref{sec:convergencetruncated}.

\begin{proposition}[Approximation is uniform in system size]\label{prop:epsilonstory}
	For all $T>0$ and for all $\kappa > 0$, there is $\varepsilon(\kappa,T)$ such that 
	\[\E \left [ \sup_{t\in[0,T]}\lVert \Pi^{N,\kappa}(t)-\Pi^{N}(t)\rVert_{\mathrm{TV}} \right ]
	\leq 
	\varepsilon(\kappa,T)\]
	and $\varepsilon(\kappa,T)\xrightarrow{\kappa\to\infty} 0$. 
\end{proposition}

\begin{proposition}[Convergence of empirical measure process in systems with freezing] \label{prop:convtrunc}
	We have $\Pi^{N,\kappa}\xRightarrow{N\to\infty}v^{\kappa}$,
	where $v^{\kappa}=(v^{\kappa}(t))_{t\geq 0}$ is the unique solution to the initial value problem: 
	$v^{\kappa}(0)=\nu\in \Mp(\N_0^{\typ})$, for $\by\in \N_0^{\typ}\setminus\kapemp$: $v^{\kappa}_{\{\by\}}(t)=\nu_{\{\by\}}(0)$; and for $\by\in \kapemp$:
	\begin{equation}
		\label{eq:limitodetruncated}
		\begin{split}
			(v_{\{\by\}}^{\kappa})'(t)=&-v_{\{\by\}}^{\kappa}(t) \sum_{i=1}^{\typ} \big(b^{i,\kappa}(\by,v^{\kappa}(t))+d^{i,\kappa}(\by,v^{\kappa}(t))+\sum_{k=1,\, k\neq i}^{\typ} m^{i,k,\kappa}(\by,v^{\kappa}(t)) \big)\\
			&+\sum_{i=1}^{\typ} \Big(v_{\{\by-\be_i\}}^{\kappa}(t)b^{i,\kappa}(\by-\be_i,v^{\kappa}(t)) + v_{\{\by+\be_i\}}^{\kappa}(t) d^{i,\kappa}(\by+\be_i,v^{\kappa}(t)) \\
			&\qquad \ + \sum_{k=1,\, k\neq i}^{\typ} v_{\{\by-\be_k+\be_i\}}^{\kappa}(t)\, m^{i,k,\kappa}(\by-\be_k+\be_i,v^{\kappa}(t))\Big).
		\end{split}
	\end{equation}
\end{proposition}

\begin{proposition}[Tightness of the empirical measure process] \label{prop:tightness}
The sequence $\{\Pi^N\}_{N\in \N}$ is tight.
\end{proposition}

We now prove Theorem~\ref{thm:conv}.
\begin{proof}[Proof of Theorem~\ref{thm:conv}]
Fix $T > 0$. 	
 By Proposition~\ref{prop:tightness}, $(\Pi^N)_{N \in \N}$ is tight. 
	Consider $(\Pi^{N_n})_{n \in \N}$ for a strictly increasing sequence $(N_n)_{n\in \N}$ in $\N$.
	There exists a weakly convergent subsequence $(\Pi^{N_{n_\ell}})_{\ell\in \N}$ and a c\`adl\`ag $\Mp(\N_0^{\typ})$-valued process $\Pi^\star$ with $\Pi^{N_{n_\ell}}\xRightarrow{\ell\to\infty} \Pi^{\star}$.
	 
  On the one hand, 
	by Proposition~\ref{prop:epsilonstory},
 \[
 \E[\sup_{t\in[0,T]}\lVert \Pi^{N_{n_\ell},\kappa}(t)-\Pi^{N_{n_\ell}}(t)\rVert_{\mathrm{TV}}]
 \leq 
 \varepsilon(\kappa,T).
 \]
	On the other hand, by Proposition~\ref{prop:convtrunc}, 
	$\Pi^{N_{n_\ell},\kappa}\xRightarrow{\ell\to\infty}v^{\kappa}$.

Let $\rho$ be the following standard metric giving the Skorohod topology on the
 space
 $D([0,T],\Mp{(\N_0^{\typ})})$ 
 of c\`adl\`ag paths from
 $[0,T]$ to $\Mp(\N_0^{\typ})$,
 \[
 \rho(\mu, \nu)
 :=
 \inf_{\lambda \in \Lambda}
 \left(\sup_{t \in [0,T]} |t - \lambda(t)|
 \vee
 \sup_{t \in [0,T]}
 \|\mu(t) - \nu \circ \lambda(t)\|_{\mathrm{TV}}\right),
 \]
where the infimum is over all continuous, increasing, bijections $\lambda: [0,T] \rightarrow [0,T]$.
(cf. equation (12.13) of \cite{Billingsley_99}.  
Let $W_1$ be the Wasserstein--$1$
metric on the space of probability measures on $D([0,T],{\Mp(\N_0^{\typ})})$ 
corresponding to $\rho$; that is,
\[
W_1(P,Q) := \inf_R \int \rho(\mu, \nu) \, R(d\mu, d\nu),
\]
where the infimum is over all probability measures $R$ on
$D([0,T],{\Mp(\N_0^{\typ})}) \times D([0,T],{\Mp(\N_0^{\typ})})$ that have respective marginals $P$ and $Q$.  Recall that $W_1$ metrizes weak convergence on the space of probability measures on $D([0,T],{\Mp(\N_0^{\typ})})$ (see, for example, Theorem 6.9 of \cite{Villani_09}).  If $\Phi$ and $\Psi$ are random elements of
$D([0,T],{\Mp(\N_0^{\typ})})$,  write $W_1(\Phi, \Psi)$ for the
Wasserstein--$1$ distance between their respective distributions.  

Observe that by setting $\Lambda(t) = t$,  we get that 
$\rho( \mu, \nu) \leq \sup_{t \in [0,T]}   \|\mu(t) - \nu(t)\|_{\mathrm{TV}},$ which implies that
\[ W_1(\Phi, \Psi) \leq \inf_R \int \sup_{t \in [0,T]}   \|\mu(t) - \nu(t)\|_{\mathrm{TV}} R(d\mu, d\nu).
\]
If $\Phi$ and $\Psi$ happen to be defined on the same probability space, we can choose as $R$ the joint distribution of $\Phi$ and $\Psi$ on that space to get
\[
W_1(\Phi, \Psi)
\le
\E \left [\sup_{t\in[0,T]}\lVert \Phi(t) -\Psi(t)\rVert_{\mathrm{TV}} \right ].
\]

Now, 
\[
\begin{split}
W_1 \left ( \Pi^{\star}, v^{\kappa} \right )\\
		\leq  & W_1 \left ( \Pi^{\star},\Pi^{N_{n_\ell}} \right ) 
		+ \E \left [ \sup_{t\in[0,T]} \lVert \Pi^{N_{n_\ell}}(t)-\Pi^{N_{n_\ell},\kappa}(t)\rVert_{TV} \right ] \\
		& \quad +  W_1 \left ( \Pi^{N_{n_\ell},\kappa}, v^{\kappa} \right ). \\
\end{split}
\]
	Taking $\ell\to\infty$ leads to the bound
	\[W_1(\Pi^{\star}, v^{\kappa}) \leq \varepsilon(\kappa,T)
 \]
	independent of the chosen subsequence $(N_{n_\ell})$. 
	Since $\varepsilon(\kappa,T)\to 0$ as $\kappa\to\infty$, we obtain 
$v^{\kappa}\xRightarrow{\kappa\to\infty}\Pi^\star$ and
 $\Pi^N\xRightarrow{N\to\infty} \Pi^\star$ upon taking $\kappa \to \infty$.
    In particular, $\Pi^\star=v$ of \eqref{eq:limitode}.
\end{proof}

\subsection{Convergence of the dynamics under freezing}
\label{sec:convergencetruncated}

To establish the convergence of the system of MTBDPs that are frozen once they reach the set of frozen states parameterized by~$\kappa$, 
we employ standard methods. 
In this regard, we rely on the following result, which is elaborated upon in \citet[Ch.~4]{Ethier1986} concerning the notation used.

\begin{theorem}\citet[Corollary 4.8.16]{Ethier1986} \label{thm:kurtzconv}
	Let $(E,r)$ be complete and separable 
	and $E_N\subset E$. 
	Let $\mathcal{A}\subset \bar{C}(E)\times \bar{C}(E)$ and 
	$v\in \Mp(E)$.
	Assume 
	\begin{enumerate}
		\item the $D(\R_+,E)$ martingale problem for $(\mathcal{A},v)$ has at most one solution, and 
		the closure of the linear span of $\mathcal{D}(\mathcal{A})$, the domain of $\mathcal{A}$, contains an algebra that separates points,
		\item for each $N\in \N$, $X_N$ is a progressively measurable process 
		with measurable contraction semigroup $\{T_N(t)\}$, 
		full generator $\hat{\mathcal{A}}_N$, 
		and sample paths in $D(\R_+, E_N)$,
		\item $\{X_N\}$ satisfies the compact containment condition; 
		that is, for every $\eta>0$ and $T>0$ 
		there is a compact set $\Gamma_{\eta,T}\subset E$ 
		such that $\inf_N \P(X_N(t)\in \Gamma_{\eta,T} \text{ for }0\leq t\leq T)\geq 1-\eta$,
		\item for each $(f,g)\in \mathcal{A}$ and $T>0$, there exists $(f_N,g_N)\in \hat{\mathcal{A}}_N$ and $G_N\subset E_N$ such that 
		$\lim_{N\to\infty} \P(X_N(t)\in G_N,\, 0\leq t\leq T)=1$, 
		$\sup_{N}\lVert f_N\rVert <\infty$, and 
		$\lim_{N\to\infty}\sup_{x\in G_N}\lVert f(x)-f_N(x)\rVert=\lim_{N\to\infty}\sup_{x\in G_N}\lVert g(x)-g_N(x)\rVert=0$,
		\item $X_N(0)\xRightarrow{N\to\infty} v$. 
	\end{enumerate} 
	Then, there exists a solution~$X$ of the $D(\R_+,E)$ martingale problem for $(\mathcal{A},v)$ 
	and $X_N\xRightarrow[]{N\to\infty}X.$ 
\end{theorem}

To apply Theorem~\ref{thm:kurtzconv}, one of the things to check is that the sequence $\{\Pi^{N,\kappa}\}_{N\in \N}$ satisfies the compact containment condition.
We will prove the following stronger result.

\begin{lemma}[Compact containment]\label{lem:compactcontainment}
	The sequences $(\Pi^N)_{N \in \N}$ and $(\Pi^{N,\kappa})_{N \in \N}$ both satisfy the compact containment condition.
\end{lemma}

\begin{proof}
	Fix $\eta,T>0$.
    By how we have coupled together the construction of the processes involved,
	we have for any $t\in[0,T]$,
	\begin{equation}
    \label{eq:monotonicity_various}
    \begin{split}
    \Pi^N(t)  \preceq \Pi^N_{\bR}(t) \preceq \Pi^N_{\bR}(T)\\
    \Pi^{N,\kappa}(t) \preceq 
    \Pi^N_{\bR}(t) \preceq \Pi^N_{\bR}(T).\\
    \end{split}
\end{equation}
 
    Recall that $r$ is the solution to the Kolmogorov forward equation of a non-explosive Markov process that is essentially a Yule process.

    Since $\Pi^N_{\bR}(T)\xrightarrow{N\to\infty}r(T)$, by Lemma~\ref{lem:convpir},
    we have that the collection of
    distributions of the sequence
    $\{\Pi^N_{\bR}(T)\}_{N\in \N}$ is tight.  Therefore, there exists a compact set $K_{\eta,T} \subseteq  \Mp(\N_0^\typ)$ such that
    $\P(\Pi_{\bR}^N(T)\in K_{\eta, T})\geq 1-\eta$ for all $N$. 

    It only remains to note that if $K$ is a compact subset of $\Mp(\N_0^{\typ})$, then so is the set $\bigcup_{\nu \in K} \{\mu 
 \in \Mp(\N_0^{\typ}) : \mu \preceq \nu\}$ and then apply
    \eqref{eq:monotonicity_various}.
\end{proof}

We are now prepared to prove the convergence of the empirical measure process in a system with freezing.

\begin{proof}[Proof of Proposition~\ref{prop:convtrunc}]
	First, we note that the initial value problem can be reduced to a finite system of ODEs.
	Its right-hand side is Lipschitz continuous because the rates can be bounded using Assumption~\ref{ass:growthlip}
	and because $\by\in \kapemp$.
	The existence and uniqueness of a solution to this system follow from classic theory (e.g. \cite[Chapter~1]{Deimling1977}).
	Note that $v^{\kappa}$ also solves (uniquely) the $(B^{N,\kappa},\nu^\kappa)$ martingale problem,
    because the martingale problem and the ODE in this 
	frozen (thus finite-dimensional) setting are equivalent \cite[Corollary 1.3]{Kurtz2011a}. 
	
	We verify that the conditions of Theorem~\ref{thm:kurtzconv} are satisfied. 
	To this end, take 
	$E_N=\Mpn(\kapemp)$ and $E=\Mp(\kapemp)$ in Theorem~\ref{thm:kurtzconv}.
        The corresponding generators that we are interested in are $B^{N,\kappa},$ as defined before Remark \ref{rem:exchangeabilitytruncated}, and $B^\kappa \coloneqq B^\kappa_\birth + B^\kappa_\death + B^\kappa_\mut$, where $$B^\kappa_\birth h (\nu) = \sum_{\by\in \kapemp}\sum_{i=1}^{\typ} \nu_{\{\by\}} b^{i,\kappa}(\by,\nu)\left [ \frac{\partial h(\nu)}{\partial \nu_{\{\by+\be_i\}}} - \frac{\partial h(\nu)}{\partial \nu_{\{\by\}}}\right ],$$ and $B^\kappa_\death$ and $B^\kappa_\mut$ are obtained similarly by modifying the definitions of $B^N_\death$ and $B^N_\mut$ before Proposition \ref{prop:exchangeability}.
        
	That there is at most one solution to the $(B^{N,\kappa},\nu^\kappa)$ martingale problem
	follows from the discussion at the beginning of this proof.
	Moreover, we have that 
	\[ \left \{ f(v)=\prod_{\by\in Y}g_\by(v_\by)\text{ with }g_\by\in \hat{C}^1(\R) \text{ and }Y\subset \N_0^{\typ}, \lvert Y\rvert <\infty \right \}\subset \bar{C}^{1}\left ( \Mp( \kapemp )\right ) \] 
	is an algebra that separates points. 
	Thus, (1) holds. 
	$\Pi^{N,\kappa}$ is an $\Mpn(\kapemp)$-valued adapted, c{\`a}dl{\`a}g Markov process and thus progressively measurable. 
	Hence, (2) holds. 
    Lemma~\ref{lem:compactcontainment} yields that (3) holds.
	For (4), fix $h\in \bar{C}^{1}( \Mp ( \kapemp ) )$. 
	Without loss of generality, 
	we can assume that \[h(\nu)=\tilde{h}(\nu_{\{\by^{(1)}\}},\ldots, \nu_{\{\by^{(k)}\}}) \]
	for some $\tilde{h}\in \bar{C}^1([0,1]^k)$ with $k\in \N$, where $\by^{(1)},\ldots,\by^{(k)}\in \kapemp$. 
	We have to find a sequence $\{h^N\}$ of functions in the domain of the generator of $\Pi^{N,\kappa}$
	that approximates $h$ (recall its form from~\eqref{eq:generatorempiricalN}; but with $f\in C((\kapemp)^N)$ because the frozen system state space is compact). 
	To this end, set
	\[\tilde{f}^N(\bz)=\tilde{h}(\pi_{\bz}( \{\by^{(1)}\} ) ,\ldots, \pi_{\bz}(\{\by^{(k)}\})) \]
	and 
 \begin{align*}
     h^N(\nu) & = \frac{1}{N!} \prod_{\substack{\bx\in \N_0^{\typ}:\\  \bx\in \supp(\pi_{\bz})}}(N\pi_{\bz}(\{\bx\}))! \sum_{\substack{\bz\in (\N_0^{\typ})^N:\\ \pi_{\bz}=\nu}} \tilde{f}^N(\bz) \\
     & =\tilde{h}(\nu_{\{\by^{(1)}\}},\ldots, \nu_{\{\by^{(k)}\}}).
 \end{align*}
	$h^N$ is in the domain of $B^{N,\kappa}$. The only difference between $h$ and $h^N$ is that $h^N$ is only defined on $\Mpn(\kapemp),$ while the domain of $h$ is $\Mp(\kapemp),$ and the two functions agree on $\Mpn(\kapemp)$.
	This implies that
	\[ \sup_{\nu\in \Mpn(\kapemp)}\lvert h(\nu)-h^N(\nu)\rvert=0,\]
        so in particular, the limit as $N \to \infty$ is 0.
	Since $\tilde{h}$ is bounded, also $\sup_{N}\lVert h^N\rVert<\infty$.
	Next, we show
	\begin{equation}
		\sup_{\nu\in \Mpn(\N_0^{\typ})} \lvert B^{N,\kappa}h^N(\nu)-B^{\kappa}h(\nu)\rvert \xrightarrow{N\to\infty}0.\label{eq:uniconvgen}
	\end{equation} 
	To this end, we start showing $\sup_{\nu\in \Mpn( \kapemp )} \lvert (B_\birth^{N,\kappa}h^N(\nu)-B_\birth^\kappa h(\nu))\rvert \xrightarrow{N\to\infty}0$. 
	Note that by Taylor's formula and~\ref{ass:growth},
	\begin{equation*}
		\begin{split}
			&\ \left \lvert B^{N,\kappa}_\birth h^N(\nu)-B_\birth^{\kappa}h(\nu) \right \rvert\\
			&=\sum_{\by\in \kapemp}\sum_{i=1}^{\typ} \nu_{\{\by\}} b^{i,\kappa}(\by,\nu)\left [N\Big(h\Big(\nu+\frac{\delta_{\by+\be_i}-\delta_{\by}}{N} \Big)-h(\nu)\Big)- \frac{\partial h(\nu)}{\partial \nu_{\{\by+\be_i\}}} +\frac{\partial h(\nu)}{\partial \nu_{\{\by\}}}\right ]\\
			&\leq \sum_{\by\in \{\tilde{\by}^{(1)},\ldots,\by^{(k)} \}}L\typ(\by_\bullet + 1 ) \cdot O(N^{-1}).
		\end{split} 
	\end{equation*}
	Since the right-hand side is independent of $\nu$,
	\[\sup_{\nu\in \Mpn( \kapemp )} \lvert B_\birth^{N, \kappa} h^N(\nu)-B_\birth^\kappa h(v)\rvert \xrightarrow{N\to\infty}0. \]
	Analogously, it can be shown that 
	$\sup_{\nu\in \Mpn(\kapemp)} \lvert B_\death^{N,\kappa} h^N-B_\death^\kappa h(\nu)\rvert \xrightarrow{N\to\infty}0$ and $\sup_{\nu\in \Mpn(\kapemp)} \lvert B_\mut^{N, \kappa} h^N-B_\mut^\kappa h(\nu)\rvert \xrightarrow{N\to\infty}0$. 
	Then (4) follows by the triangle inequality.
	By assumption, (5) holds. 
	In particular, we have checked (1)--(5) of Theorem~\ref{thm:kurtzconv} and thus the result follows.
\end{proof}

Next, we prove the bound on $\E[\sup_{t\in[0,T]}\lVert \Pi^{N,\kappa}(t)-\Pi^{N}(t)\rVert_{\mathrm{TV}}]$ that is uniform in $N$.

\begin{proof}[Proof of Proposition~\ref{prop:epsilonstory}]	
	Fix $T>0$. 
	For $u \leq \inf\{s\geq 0: \bulletize{\bZ^N_j(s)} = \kappa\}$, we have $\bZ^N_j(u)=\bZ^{N,\kappa}_j(u)$, $j \in [N]$.
	Thus, for $0 \leq t \leq T$, 
 $\lvert\{j \in N:\bZ^N_j(t)\neq\bZ^{N,\kappa}_j(t)\}\rvert
 \leq 
 \lvert\{j \in [N] :\bR^N_j(t) \notin \kapemp)\}\rvert 
 \leq
 \lvert\{j \in [N] :\bR^N_j(T) \notin \kapemp)\}\rvert
 $.
	In words: 
	the count of families that have different compositions under the original and frozen dynamics 
	is bounded from above by 
	the count of families in the dominating pure-birth-type process 
	that exited $\kapemp$. 
 Thus,
 \[
 \begin{split}
& \E[\sup_{t\in[0,T]}\lVert \Pi^{N,\kappa}(t)-\Pi^{N}(t)\rVert_{\mathrm{TV}}] \\
& \quad \le \E\left[\sup_{t\in[0,T]} \frac{1}{N} \lvert\{j \in N:\bZ^N_j(t)\neq\bZ^{N,\kappa}_j(t)\}\rvert \right] \\
& \quad \le \E\left[ \frac{1}{N} \lvert\{j \in [N] :\bR^N_j(T) \notin \kapemp)\} \right] \\
& \quad = \P(\bR_1^{N}(T) \notin \kapemp). \\
\end{split} 
\]

 We know, however, from Lemma~\ref{lem:convpir}(i) that the sequence $\{\bR_1^N(T)\}_{N \in \N}$ is weakly convergent and hence tight, so $\varepsilon(\kappa,T) := \sup_{N \in \N} \P(\bR_1^{N}(T) \notin \kapemp)$ has the desired properties.
 
\end{proof}

We now address the tightness of the sequence $\{\Pi^N\}_{N \in \N}$.

\begin{proof}[Proof of Proposition~\ref{prop:tightness}]
	From Theorems 23.8 and 23.11 of \cite{Kallenberg2021}, it suffices to check the following.
 \begin{enumerate}
		\item For every $\eta>0$ and $T>0$ 
		there is a compact set $\Gamma_{\eta,T}\subset \Mp(\N_0^{\typ})$
		such that $\inf_N \P(\Pi^N(t)\in \Gamma_{\eta,T} \text{ for }0\leq t\leq T)\geq 1-\eta$.
		\item For all $T > 0$,  
		\begin{equation*}
			\lim_{\theta\searrow 0} \sup_{N \in \N} \sup_{S \in \mathcal{S}_T^N} \sup_{0 \le u \le \theta}\E[\lVert \Pi^N(S+u)-\Pi^{N}(S)\rVert_{TV} ] =0,
		\end{equation*}
	\end{enumerate}
 where $\mathcal{S}_T^N$ is the set of all discrete $\sigma(\Pi^N)$-stopping times that are bounded by~$T$. 
 
 Part 1 has been verified in Lemma~\ref{lem:compactcontainment}.
 
 For Part $2$, note that a.s.
	\begin{align*}
		& \lVert \Pi^N(S+u)-\Pi^N(S)\rVert_{TV}\\ 
  & \quad =\frac{1}{2}\frac{1}{N}\sum_{\by\in \N_0^{\typ}} \lvert \lvert \{j \in [N]:\bZ_j^N(S+u)=\by\}\rvert-\lvert \{j \in [N] :\bZ_j^N(S)=\by\}\rvert\rvert \\
		& \quad \leq \frac{1}{2} \frac{1}{N} \sum_{\by\in \N_0^{\typ}}\sum_{j=1}^N  \lvert \ind_{\{\bZ^N_j(S+u)=\by\}}-\ind_{\{\bZ^N_j(S)=\by\}}\rvert \\
		& \quad \leq  \frac{1}{N} \sum_{j=1}^N   \ind_{\{\bZ^N_j(S+u)\neq\bZ^N_j(S)\}}.
	\end{align*}
	In particular, using exchangeability, for $u \in [0,\theta]$ and $S \in \mathcal{S}_T^N$
	\begin{align*}
		\E[\lVert\Pi^N(S+u)-\Pi^N(S)\rVert_{TV}]
  &  \leq \P(\bZ_1^N(S+u)\neq \bZ_1^N(S)) \\
  &  \leq \P( \|\bZ_1^N(S+u)-\bZ_1^N(S)\|_1 \geq\varepsilon) \\
  & \leq \P( \bulletize{\bR_1^N(S+u)}-\bulletize{\bR_1^N(S)} \geq\varepsilon) \\
  & \leq \P( \bulletize{\bR_1^N(S+\theta)}-\bulletize{\bR_1^N(S)} \geq\varepsilon) \\
  & \leq \P( \bulletize{\bR_1^N(T+\theta)}-\bulletize{\bR_1^N(T)} \geq\varepsilon)
	\end{align*}
        for any $\varepsilon>0$.
        For all $N \in \N$, $\lim_{\theta \searrow 0}(\bulletize{\bR_1^N(T+\theta)}-\bulletize{\bR_1^N(T)}) = 0$ almost surely.  Also, by Lemma~\ref{lem:convpir}(i), $\bulletize{\bR_1^N(T+\theta)}-\bulletize{\bR_1^N(T)}$ converges in distribution to $\bulletize{\bR_1^\infty(T+\theta)}-\bulletize{\bR_1^\infty(T)}$ and $\lim_{\theta \searrow 0}(\bulletize{\bR_1^\infty(T+\theta)}-\bulletize{\bR_1^\infty(T)}) = 0$ almost surely.  Combining these observations gives Part 2.
\end{proof}

Finally, we address the convergence
of the sequence $(\bZ_1^N, \ldots, \bZ_k^N)_{N \in \N}$ for each fixed $k \in \N$. 

\begin{lemma}\label{lem:tightnessZ1}
For each $k \in \N$, the sequence 
$\{(\bZ^N_1, \ldots, \bZ_k^N)\}_{N\in \N}$ is tight.
\end{lemma}

\begin{proof}
 From Lemma~\ref{lem:dominationbirth} we know that 
 almost surely for all $0 \le s< t$, all $N \in \N$, $j \in [N]$, and $i \in [D]$, that $Z_{j,i}^N(t) \le R_{j,i}^N(t)$
and
$|Z_{j,i}^N(t) - Z_{j,i}^N(s)| \le |R_{j,i}^N(t) - R_{j,i}^N(s)|$.
It follows from these comparisons, the necessary and sufficient conditions for tightness in Theorem~7.2 in Chapter~4 of \cite{Ethier1986}, and the convergence (hence tightness)  of the sequence $\{(\bR_1^N, \ldots, \bR_k^N)\}_{N \in \N}$ established in Lemma~\ref{lem:convpir}, that the sequence $\{(\bZ_1^N, \ldots, \bZ_k^N)\}_{N \in \N}$ is tight.
\end{proof}

\begin{proof}[Proof of Corollary~\ref{cor:conv}]
For ease of notation, set
$\bZ_{[k]}^N := (\bZ_1^N, \ldots, \bZ_k^N)$, $N \in \N$.
We know from Remark~\ref{rem:chaos_implication} that $(\bZ_{[k]}^N(0))_{N \in \N}$ converges in distribution to a random element with distribution $\nu^{\otimes k}$.

By Lemma~\ref{lem:tightnessZ1}, the sequence 
$(\bZ_{[k]}^N)_{N \in \N}$ is tight.

Note for any function $f \in C_c((\N_0^\typ)^k)$ that 
\[
\begin{split}
& f(\bZ_{[k]}^N(t))\\
&-
\int_0^t
\sum_{j \in [k]}
\biggl [
\sum_{i \in [\typ]} b^i(\bZ_j^N(s), \Pi^N(s)) (f(\bZ_{[k]}^N(s) + e_{j,i}) - f(\bZ_{[k]}^N(s)))\\
&
+
\sum_{i \in [\typ]} d^i(\bZ_j^N(s), \Pi^N(s)) (f(\bZ_{[k]}^N(s) - e_{j,i}) - f(\bZ_{[k]}^N(s)))\\
&
+
\sum_{i, \ell \in [\typ], \ell \ne i} m^{i,\ell}(\bZ_j^N(s), \Pi^N(s)) (f(\bZ_{[k]}^N(s) + e_{j,\ell} - e_{j,i}) - f(\bZ_{[k]}^N(s)))
\biggr ] \, ds \\
\end{split}
\]
is a martingale. 

From Theorem~\ref{thm:conv}, we have that $(\Pi^N)_{N \in \N}$ converges in probability to $v$, so any subsequential limit $\bZ_{[k]} := (\bZ_1^\infty, \ldots, \bZ_k^\infty)$ is such that for 
any function $f \in C_c((\N_0^\typ)^k)$,
\[
\begin{split}
& f(\bZ_{[k]}^\infty(t))\\
&-
\int_0^t
\sum_{j \in [k]}
\biggl [
\sum_{i \in [\typ]} b^i(\bZ_j^\infty(s), v(s)) (f(\bZ_{[k]}^\infty(s) + e_{j,i}) - f(\bZ_{[k]}^\infty(s)))\\
&
+
\sum_{i \in [\typ]} d^i(\bZ_j^\infty(s), v(s)) (f(\bZ_{[k]}^\infty (s) - e_{j,i}) - f(\bZ_{[k]}^\infty(s)))\\
&
+
\sum_{i, \ell \in [\typ], \ell \ne i} m^{i,\ell}(\bZ_j^\infty(s), v(s)) (f(\bZ_{[k]}^\infty(s) + e_{j,\ell} - e_{j,i}) - f(\bZ_{[k]}^\infty(s)))
\biggr ] \, ds \\
\end{split}
\]
is a martingale (comparisons with $(\bR^N)$ establish the necessary uniform integrability).  This completes the proof.
\end{proof}

\section{A computationally tractable special case: locally simple with moment-mediated interactions}
\label{sec:examples}

In this section, we study a special case of the mean-field interacting MTBDP defined in \S\ref{sec:theory} that is amenable to calculations in the context of a phylogenetic birth-death model.
We specialize to a case with no local interactions, and with global interactions mediated by moments of the limiting transition probability $v(t)$ defined by \eqref{eq:limitode}.
This class of processes is rich enough to model both carrying capacity and frequency-dependent selection, and does not add undue computational complexity. 

\begin{example}[Simple MTBDP with moment-mediated mean-field interactive death rates]
\label{ex:linearmtbd}

Consider the MTBDP with transition rates \begin{align*}
	b^i(\by, \nu) = y_i \lambda_i,\quad
	d^i(\by, \nu) = y_i \tilde\mu_i\left(\textstyle\sum_{\by\in \N_0^D} \by \nu_{\{\by\}} \right),\quad 
    m^{i,j}(\by, \nu) = y_i \Gamma_{i,j},
\end{align*}
$i,j\in[D],\ j\neq i$, where $\blambda\in \nnR^{\typ}$, $\Gamma \in \R^{\typ \times \typ}$ (with $\Gamma\bone = \bzero$, and non-negative off-diagonal entries), and $\tilde\bmu\in C(\R_+^D,\R_+^D)$.
Set $\br(t)\coloneqq \sum_{\by\in\N_0^D} \by \, v_{\{\by\}}(t)$, where $(v(t))_{t\ge 0}$, the solution to~\eqref{eq:limitode}, is the limit of the empirical distribution processes $(\Pi^N(t))_{t \ge 0}$. Then $\br=(\br(t))_{t\geq 0}$ is the first-moment process and it solves the finite, closed system, \begin{equation}
	\label{eq:moment}
	\begin{split}
		r_i(t)' &= (\lambda_i - \tilde\mu_i(\br(t))) \, r_i(t) + \sum_{j=1}^{\typ}\Gamma_{ji} r_j(t), \quad i=1,\dots,\typ\\
		\br(0) &= \br_0,
	\end{split}
\end{equation}
where $\br_0$ is the expected initial state. 
Note that a solution of \eqref{eq:limitode} has finite expectation, via Lemma~\ref{lem:dominationbirth}.
\end{example}
Example~\ref{ex:linearmtbd} specializes the general mean-field interactions considered in Section~\ref{sec:theory} to a mean-field interaction mediated by the expected state vector (the first moment of the state distribution).
In that case, the interaction field is the solution of the finite-dimensional nonlinear moment equation \eqref{eq:moment}, so we can bypass solving the full infinite-dimensional nonlinear forward equation \eqref{eq:limitode}.

\begin{example}[Linear moment interaction]
	\label{ex:linear}
	As a simple and biologically interpretable example of the special case of Example~\ref{ex:linearmtbd}, we take $\tilde\bmu(\br(t)) = \bmu + W \br(t)$, where $\bmu \in \nnR^\typ$ is constant and the matrix $W\in\nnR^{\typ \times \typ}$ parameterizes the interaction.
	If $W$ is the matrix of $1$s, then each element of the $\typ$-vector $W \br(t)$ is the expected total population size of the focal process at time $t$, and death rates increase with this total size, enforcing a carrying capacity.
	Otherwise, the death rates are also sensitive to the expected relative frequency of each type.
	For example, if $W$ is diagonally dominant, then the model includes negative frequency-dependent selection.	
	Technically, to satisfy Assumption~\ref{ass:growthlip}, we require that this linear term is truncated above some value of the expected total size.
    In practice, we take this cut-off to be very large, such that the numerical results below are not impacted.
\end{example}

\subsection{Steady states induced by mean-field interaction}

While the simple MTBDP displays only trivial steady states (or constant ones, if it is critical), the MTBDP with mean-field interaction admits more interesting behavior.
Steady-state behavior can be examined by imposing a criticality condition on the self-consistent field.
For Example~\ref{ex:linear}, steady states $\br^\ast\in\nnR^\typ$ satisfy
\begin{equation}
	\label{eq:criticality}
	\left(\text{diag}(\blambda - \bmu - W\br^\ast) + \Gamma\right)\br^\ast = 0.
\end{equation}
Nontrivial solutions of this system of nonlinear algebraic equations for the critical field can be found numerically with standard root-finding methods, and are indeed steady states as long as the process is supercritical when the field vanishes.
For results on steady-state solutions in strongly interacting MTBDPs, see \citet{dessalles2018exact}.

\subsection{Numerical examples}

\begin{figure}[ht]
	\centering
	\includegraphics[width=1.\textwidth]{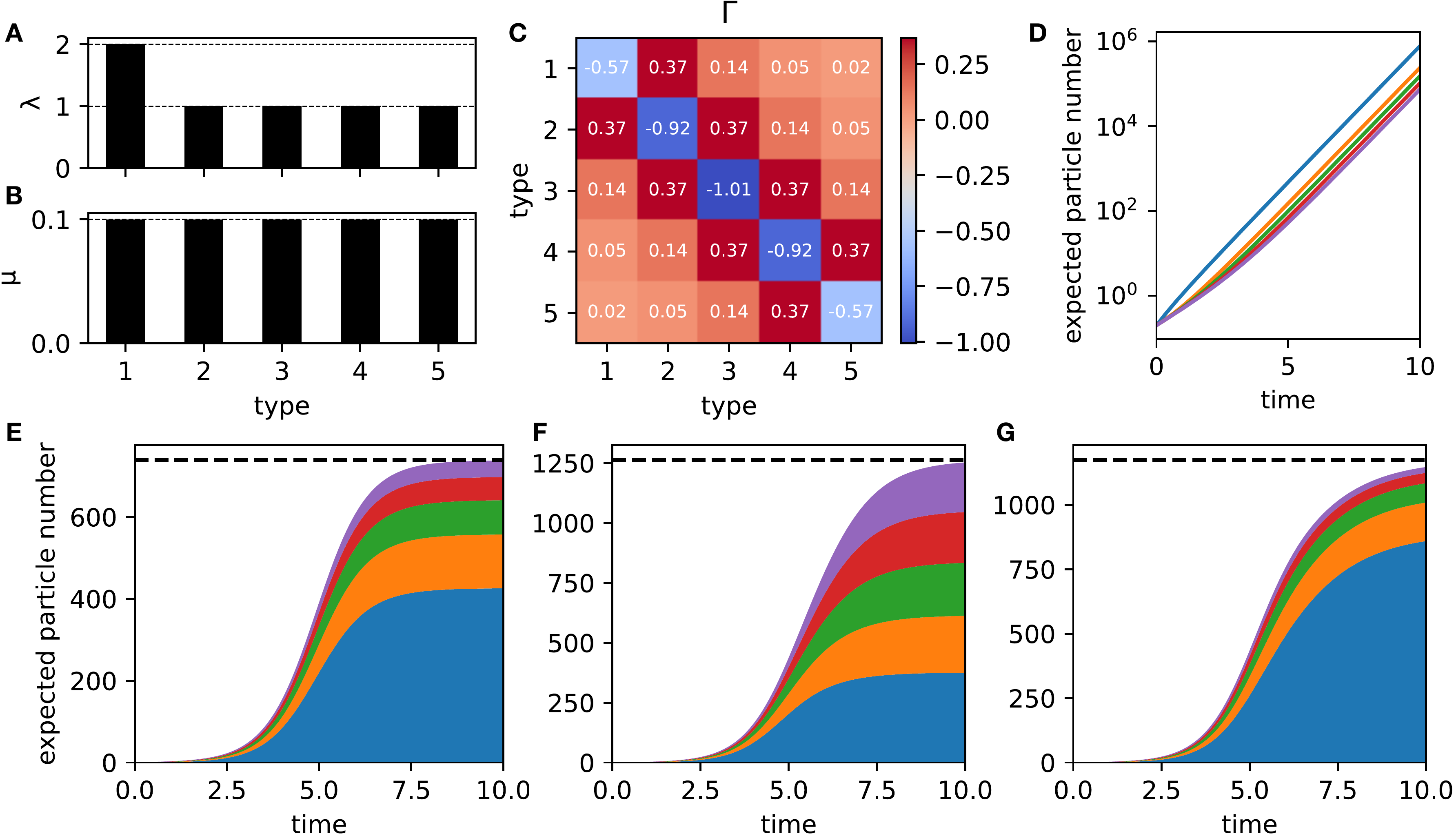}
	\caption{
	Numerical solutions of the self-consistent field $\br$ for the MTBDP with moment-mediated mean-field interaction (Example~\ref{ex:linear}).
	In these examples, $\typ=5$.
	{(\bf A.)} Birth rates $\blambda$, with type 1 elevated above the others.
	{(\bf B.)} Death rate component $\bmu$, the same for all types.
	{(\bf C.)} Type transition rate matrix $\Gamma$, of Toeplitz form, so that mutations between neighboring states are more likely.
	{(\bf D.)} Expected particle count of each of the 5 types (colors) in the absence of any mean-field interaction, showing supercritical growth.
	{\bf E-F} show stacked expected particle count of each of the 5 types, with various mean-field interactions of the form Example~\ref{ex:linear}, with $\|W\|_{\mathrm{F}} = 0.01$ in all cases.
	The dashed lines indicate the critical fields $\br^\ast$ computed by solving \eqref{eq:criticality}.
	{(\bf E.)} Carrying capacity: $W \, \propto \, J$ (with $J$ denoting the $\typ \times \typ$ matrix of $1$s).
	{(\bf F.)} Negative frequency-dependent selection: $W \, \propto \, I$.
	{(\bf G.)} Positive frequency-dependent selection: $W \, \propto \, J - \sfrac{3}{5} \, I$.
	}
	\label{fig:moments}
\end{figure}

For Example~\ref{ex:linear}, the nonlinear moment equation~\eqref{eq:moment} is of Ricatti type, with only quadratic nonlinearities.
Figure~\ref{fig:moments} shows numerical results for the self-consistent field $\br$ of Example~\ref{ex:linear} with $\typ=5$ types.
The field in this case represents the vector of expected particle counts over the 5 types.
These three examples model carrying capacity, negative frequency-dependent selection, and positive frequency-dependent selection, and all use the same rate parameters $\blambda, \bmu, \Gamma$ (Figure~\ref{fig:moments}A-C) but different $W$ matrices.
Without mean-field interaction ($W = 0$), this MTBDP is supercritical, and the expected particle counts grow exponentially (Figure~\ref{fig:moments}D).
One particle type has a higher birth rate than the others, so it grows faster.
With a carrying capacity interaction (Figure~\ref{fig:moments}E), the population reaches a stationary phase due to a mean-field interaction that increases death rates linearly with the expected population size.
With negative frequency-dependent selection (via a diagonally dominant $W$), the types are more balanced (Figure~\ref{fig:moments}F) because the death rate of a given type is suppressed only by growth of that type.
With positive frequency-dependent selection (via a diagonally non-dominant $W$), the death rate of a given type is less suppressed by growth of that type than the others, leading to an enhancing effect on the type with the birth rate advantage (Figure~\ref{fig:moments}G).

A Python implementation producing the results above is available at \url{https://github.com/WSDeWitt/mfbd}.
This code is written in JAX \citep{jax2018github} and relies on the Diffrax package \citep{kidger2021on} for numerical ODE solutions.
Specifically, to solve Riccati-type ODEs we use the Dormand-Prince 8/7 method \citep{prince1981high}---a high-accuracy explicit Runge-Kutta solver---with Hermite interpolation for dense evaluation in the time domain.
To adapt step sizes we use an I-controller \citep[see][\S II.4]{hairer2008solving-i}.
To solve the nonlinear algebraic equations for the critical field, we use root finding with automatic differentiation in JAXopt \citep{jaxopt_implicit_diff}.

\subsection{Integrating mean-field interactions in phylogenetic birth-death models}

Phylogenetic birth-death models augment the simple MTBDP with a sampling process that results in partially observed histories, and are considered as generative models for phylogenetic trees.
They add two additional parameters: the sampling probability $\rho$ gives the probability that any given particle at a specified final sampling time (the present) is sampled, and the fossilization probability $\sigma$ gives the probability that a death event before the present is observed.
The tree is then partially observed by pruning out all subtrees that are not ancestral to a sampled tip or fossil.

Computing likelihoods for rate parameters on phylogenetic trees requires marginalizing out all possible unobserved sub-histories, conditioned on the partially observed history.
We briefly outline this calculation, augmented with mean-field interactions.
We use notation like that of \citet{Kuhnert2016} and \citet{BaridoSottani2018}.
Given the parameters for the system in Example~\ref{ex:linearmtbd}, and measuring time backward from the present sampling time, the probability density requires solving three coupled initial value problems (the standard case without mean-field interactions solves two systems).

First, the self-consistent fields $\br(t)$ are calculated as in Example~\ref{ex:linear} by solving a $\typ$-dimensional initial value problems (we reverse time such that the process starts at the tree root time $\tau>0$, and ends at $t=0$).
Next, we need as an auxiliary calculation the probability $p_i(t)$ that a particle of type $i$ at time $t$ (before the present) will not be observed in the tree---that is, it will not be sampled and will not fossilize.
These are given by the system of backward equations (of Riccati type)
\begin{equation}
	\label{eq:nonobs}
	\begin{split}
	p_i'(t) &= \lambda_i p_i(t)^2 - \left(\lambda_i + \mu_i + \sum_{j=1}^\typ W_{ij}r_j(t)\right) p_i(t)\\
			 & \quad + \sum_{j=1}^\typ \Gamma_{ij} p_j(t) + (1-\sigma)\left(\mu_i + \sum_{j=1}^\typ W_{ij}r_j(t)\right)\\
	p_i(0) &= 1-\rho,
	\end{split}
\end{equation}
where $\br$ is given as in Example \ref{ex:linearmtbd}.
These are solved on the interval $[0, \tau]$ where $\tau$ is the age of the root of the tree.

Finally, we compute the likelihood contribution for each of $B$ tree branches $b=1,\dots, B \in \N$.
Fixing some branch $b$ with type $i$ spanning the half-open interval $(t_1, t_2]$, let $q_{i}(t)$ denote its branch propagator, defined as the solution of the backward equation
\begin{equation}
	\label{eq:branchprop}
	\begin{split}
	q'_i(t)& = \left(2\lambda_i p_i(t) + \Gamma_{ii} - \lambda_i - \mu_i - \sum_{j=1}^\typ W_{ij}r_j(t)\right)q_i(t)\\
	q_i(t_1)& =
	\begin{cases}
		\rho, \quad &\text{if branch $b$ leads to a sample at $t_1=0$}\\
		\sigma\mu_i, \quad &\text{if branch $b$ leads to a fossil at $t_1>0$}\\
		\lambda_i q_{\mathrm{left}}(t_1)q_{\mathrm{right}}(t_1), \quad &\text{if branch $b$ splits at time $t_1>0$}\\
		\Gamma_{ij}q_j(t_1), \quad &\text{if branch $b$ transitions to type $j$ at time $t_1>0$}
	\end{cases}
\end{split}
\end{equation}
where $q_{\mathrm{left}}$ and $q_{\mathrm{right}}$ denote the propagators of the left and right children of branch $b$.
This system is coupled via the boundary conditions for each branch, and can be solved recursively by post-order tree traversal, yielding the tree likelihood accumulated at the root.

Standard phylogenetic birth-death models are recovered by setting $W=0$, and only solving the $\bp$ and $\bq$ systems.
By solving $\bp$, $\bq$, and $\br$ systems in the case $W\ne 0$, it is possible to compute tree likelihoods under phylogenetic birth-death processes that model interactions, while maintaining the efficient post-order calculation of likelihoods.

\section{Discussion} \label{sec:discussion}

Incorporating interactions in birth-death processes is challenging for inference applications, but we summarize some developments.
\citet{Crawford2014} developed techniques based on continued fraction representations of Laplace convolutions to calculate transition probabilities for general single-type birth-death processes, without state space truncation.
\citet{ho2018birth} calculate transition probabilities for the \emph{birth/birth-death} process---a restricted bivariate case where the death rate of one type vanishes, but rates may be otherwise nonlinear.
\citet{Xu2015} use branching process approximations of birth-death processes and generating-function machinery \citep{wilf2005generatingfunctionology} for moment estimation.
\citet{casanova2021} study single-type branching processes with strong interactions, restricted to a regime in which duality methods can be used to characterize the stationary distribution.

Instead of the strong interactions considered in the above work, we have introduced an MTBDP with mean-field interactions.
This mean-field system restores (in the limit) the computational tractability of the non-interacting case.
We have established the fairly general conditions under which this process is well-defined, demonstrated how to perform mean-field calculations in the context of a phylogenetic birth-death model, and provided an efficient software implementation.
While we were motivated by evolutionary dynamics of antibodies in germinal centers, we also foresee applications to other somatic evolution settings, such as tumor evolution and developmental lineage tracing, and to experimental microbial evolution.
While we have outlined how to evaluate likelihoods for phylogenetic birth-death models with mean-field interaction, we leave inference on biological data for future work.

The moment-mediated interactions we study allow for direct solution of self-consistent fields via a nonlinear moment equation, which is amenable to standard numerical ODE techniques.
Mean-field calculations in physical applications (typically on continuous spaces with nonlinear PDEs) often rely on the \emph{self-consistent field method}, which solves a sequence of linear systems with an external field that converges to a fixed point (for example, the Hartree-Fock, and density-functional theories for quantum many-body systems) \citep{yasodharan2021four, RevModPhys.68.13}.
Such methods tacitly assume a contractive mapping holds for this procedure, so that, by the Banach Fixed-Point Theorem, the field converges to a unique point.
In practice, the method can suffer from slow convergence, non-convergence, or even divergence of the iterates, although there are several regularization techniques for controlling these issues.
Our direct solution for the moment-mediated case avoids these issues.

We have suppressed explicit time dependence in the particle-wise birth and death rates $\blambda$ and $\bmu$ for notational compactness, but all the results of \S\ref{sec:theory} and \S\ref{sec:examples} extend to the inhomogeneous case $\blambda(t)$ and $\bmu(t)$ with suitable continuity assumptions in the time domain.
We note, however, that our mean-field approach involves effective time-dependence in the rates even if the intrinsic rates are not explicitly time-dependent.
This effective time-dependence arises from specifying a finite number of dynamical parameters (i.e., the rates and the interaction matrix $W$) that uniquely determine an effective field via the condition of self-consistency, Theorem~\ref{thm:conv}.

Finally, we notice that our mean-field system of $N$ interacting replica trees has a self-similarity property: if we consider a subset of $N$ particles from one of the replicas at time $t>0$, this looks like the starting configuration of a new $N$-system.
This suggests that our mean-field model could also be used as an approximation for strong interactions within a single MTBDP.
However, the appropriate notions of convergence and exchangeability are less clear in this case.
The validity of a mean-field approximation for a single self-interacting MTBDP would seem to involve a delicate balance of quenched disorder from early times when the process is small, on the one hand, with the limiting mean-field interaction when the process is large, on the other hand.
We save these questions for future work.

\section*{Acknowledgments}

WSD thanks Erick Matsen and Gabriel Victora for discussions on carrying capacity and competitive interactions in germinal center evolution, Yun Song for suggesting fixed-point methods for self-consistency calculations, and Volodymyr Minin for discussions on moment-based techniques for birth-death processes.
We are grateful to the anonymous reviewers for several suggestions which improved the exposition.
WSD was supported by a Fellowship in Understanding Dynamic and Multi-scale Systems from the James S. McDonnell Foundation.
EH was funded by the Citadel Fellowship of the Statistics Department at the University 
of California at Berkeley.
SH was funded by the Deutsche Forschungsgemeinschaft (DFG, German Research Foundation) -- Projektnummer 449823447.

\addtocontents{toc}{\protect\setcounter{tocdepth}{2}}
\bibliographystyle{elsarticle-harv}
\bibliography{reference}

\end{document}